\documentclass[english]{article}
\usepackage[latin9]{inputenc}
\usepackage{mathrsfs}
\usepackage{amsmath}
\usepackage{amssymb}
\usepackage{esint}

\makeatletter

\title{Newtheorem and theoremstyle test}
\author{Michael Downes\\updated by Barbara Beeton}

\usepackage[exercise]{amsthm}

\newtheorem{thm}{Theorem}[section]
\newtheorem{cor}[thm]{Corollary}

\newtheorem{lem}[thm]{Lemma}

\theoremstyle{remark}
\newtheorem*{rmk}{Remark}

\theoremstyle{plain}
\newtheorem{Def}{Definition}

\newtheoremstyle{note}
  {3pt}
  {3pt}
  {}
  {}
  {\itshape}
  {:}
  {.5em}
  {}

\theoremstyle{note}

\newtheoremstyle{citing}
  {3pt}
  {3pt}
  {\itshape}
  {}
  {\bfseries}
  {.}
  {.5em}
  {\thmnote{#3}}

\theoremstyle{citing}

\newtheoremstyle{break}
  {9pt}
  {9pt}
  {\itshape}
  {}
  {\bfseries}
  {.}
  {\newline}
  {}

\theoremstyle{break}

\theoremstyle{exercise}

\swapnumbers
\theoremstyle{plain}

\let\lvert=|\let\rvert=|

\usepackage{babel}

\usepackage{babel}

\usepackage{babel}

\makeatother

\usepackage{babel}
\begin{document}

\title{On the radius pinching estimate and uniqueness of the CMC foliation
in asymptotically flat 3-manifolds}

\author{Shiguang Ma}
\maketitle
\begin{abstract}
In this paper we consider the uniqueness problem of the constant mean
curvature spheres in asymptotically flat 3-manifolds. We require the
metric have the form $g_{ij}=\delta_{ij}+h_{ij}$ with $h_{ij}=O_{4}(r^{-1})$
and $R=O(r^{-3-\tau}),\tau>0$ . We do not require the metric to be
close to Schwarzschild metric in any sense or to satisfy RT conditions.
We prove that, when the mass is not $0$, stable CMC spheres that
separate a certain compact part from infinity satisfy the radius pinching
estimate $r_{1}\leq Cr_{0}$, which in many cases is critical to prove
the uniqueness of the CMC spheres. As applications of this estimate,
we remove the radius conditions of the uniqueness result in \cite{Huang-CMC}
and \cite{NERZ-CMC} in some special cases.
\end{abstract}

\section{Introduction}

In Generality Relativity we usually study the asymptotically flat
3-manifolds. It can be considered as the initial data set of the Einstein
Equation. To study the geometry of such manifolds is interesting and
useful. In 1996 Huisken and Yau proved in \cite{Huisken-Yau} that
in asymptotically Schwarzschild manifolds with positive mass, there
exists a foliation of strictly stable constant mean curvature(CMC)
spheres. They also used this foliation to define the center of mass
of the asymptotically flat manifolds. The uniqueness of such CMC foliation
is a harder problem. If this CMC foliation is unique, it can be regarded
as a canonical object of the asymptotically flat end. Actually such
CMC foliation is proposed to be the abstract definition of the center
of mass of the asymptotically flat manifolds. Huisken and Yau proved
that for $1/2<q\leq1$, stable CMC sphere outside $B_{H^{-q}}(0)$
is unique, where $H$ is the constant mean curvature of the sphere.
In 2002, Jie Qing and Gang Tian removed this radius condition and
proved a sharper uniqueness theorem in \cite{Qing-Tian-CMC}. They
used a scaling invariant integral to detect the positive mass. To
calculate this integral they did blow-down analysis on the constant
mean curvature spheres in three different scales and used some technique
from harmonic maps to deal with the intermediate scale. Then in \cite{Huang-center-of-mass},
Lan-hsuan Huang considered the general asymptotically flat manifolds
with Regge-Teitelboim condition. She proved a similar result as Huisken
and Yau. Her uniqueness result also needs radius condition $\ensuremath{r_{1}\leq C_{1}r_{0}^{\frac{1}{a}}}$
for some $a$ satisfying $\ensuremath{\frac{5-q}{2(2+q)}<a\leq1}$
where 
\begin{align}
r_{0}=\inf\{|x|;x\in\Sigma\}, & r_{1}=\sup\{|x|;x\in\Sigma\},\label{inner and outer radius}
\end{align}
where $|x|=\sqrt{x_{1}^{2}+x_{2}^{2}+x_{3}^{2}}$ and $\Sigma$ is
the constant mean curvature sphere. In \cite{Eichmair-Metzger-iso-surface,Eichmair-Metzger-iso-all-dim},
Eichmair and Metzger considered the existence and uniqueness of isoperimetric
surfaces in asymptotically flat manifolds which are $C^{0}$ asymptotic
to Schwarzschild manifolds (for uniqueness they require more smoothness).
However, their uniqueness result is in the class of isoperimetric
surfaces which is stronger than the class of stable constant mean
curvature surfaces. In \cite{Shiguang-Ma-CMC} I studied the uniqueness
problem in $(m,k,\varepsilon)$-AF-RT manifolds which requires the
manifolds to be close to asymptotically Schwarzschild manifolds in
some Sobolev space and under the weaker radius condition $\log(r_{1})\leq Cr_{0}^{1/4}$
I proved the uniqueness of the stable CMC spheres outside a certain
compact set. Recently Christopher Nerz announced a result on the existence
and the uniqueness of CMC foliation in asymptotically flat manifolds
without RT conditions. The AF manifolds he studied is $C_{\frac{1}{2}+\varepsilon}^{2}$
with non 0 mass. His uniqueness result also requires radius conditions,
i.e. the CMC surfaces lies in the class $A_{r}^{\varepsilon,\eta}(C_{0},C_{1}).$
This condition implies 
\[
r_{1}(\Sigma)\leq Cr_{0}(\Sigma)
\]
for some $C>0.$ 

In sum, most of the above theorems (expect Qing and Tian's result)
need certain type of radius condition. $r_{0}(\Sigma)\geq Cr_{1}(\Sigma)$
is stronger than the radius conditions used in \cite{Huisken-Yau}
and in \cite{Huang-CMC}. In this paper, we prove that in a large
kind of asymptotically flat manifolds, stable CMC spheres that separate
a certain compact part from infinity will satisfy $r_{1}(\Sigma)\leq Cr_{0}(\Sigma)$
automatically. We don't need the manifolds to be close to Schwarzchild
and to satisfy RT conditions. 

\begin{Def}A 3-manifold $M$ with a Riemannian metric $g$ is called
$C_{q,\tau}^{4}$-asymptotically flat with $q\in(\frac{1}{2},1]$
if there is a compact set $K^{\prime}\subset M$ such that $M\backslash K^{\prime}$
is diffeomorphic to $\mathbb{R}^{3}\backslash B_{1}(0)$ and in the
Euclidean coordinates $\{x_{i}\}_{i=1}^{3}$, the metric takes the
form 
\[
g_{ij}(x)=\delta_{ij}+h_{ij}(x)
\]
where $h_{ij}=O_{4}(r^{-q})$ and the scalar curvature of $g_{ij}$
satisfies $R=O(r^{-3-\tau})$. Here, $f=O_{k}(|x|^{-q})$ means $\partial^{l}f=O(|x|^{-l-q})$
for $l=0,\cdots,k$.

\end{Def}

Sometimes we call $M\backslash K'$ an asymptotically flat end. 

In such $C_{q,\tau}^{4}$-AF manifolds, as the scalar curvature $R$
is integrable, one can define the mass, see \cite{Bartnik-mass}. 

\[
m=\lim_{r\rightarrow\infty}\frac{1}{16\pi}\int_{|x|=r}(h_{ij,j}-h_{jj,i})v_{g}^{i}d\mu_{g},
\]
 where $v_{g}$ and $d\mu_{g}$ are the unit normal vector and volume
form with respect to the metric $g$. In this paper, we omit the subscript
$g$ when we work in metric $g$ and we will not omit the subscript
$e$ when we work in Euclidean metric. 

Let $\Sigma$ be a constant mean curvature (CMC for short) surface.
We say it is stable if the second variation operator has only non-negative
eigenvalues when restricted to the functions with $0$ mean value,
i.e. 
\begin{equation}
\int_{\Sigma}(|A|^{2}+Ric(v,v))f^{2}d\mu\leq\int_{\Sigma}|\nabla f|^{2}d\mu
\end{equation}
 holds for function $f$ with $\int_{\Sigma}fd\mu=0$, where $A$
is the second fundamental form, and $Ric(v,v)$ is the Ricci curvature
in the normal direction with respect to the metric $g.$

Now let's state the main theorem of this paper.

\begin{thm}\label{thm1} Let $(M,g)$ be $C_{1,\tau}^{4}$-AF with
non 0 mass . Then there is a compact set $\tilde{K}\subset M$ and
$C>0$ such that any stable CMC sphere $\Sigma$ which separates $\tilde{K}$
from the infinity has 
\[
r_{1}(\Sigma)/r_{0}(\Sigma)\leq C
\]
 with $r_{0}$ and $r_{1}$ defined by (\ref{inner and outer radius}).

Moreover, for a sequence of stable CMC sphere $\Sigma_{n}$ which
separate $K'$ from infinity with 
\[
\lim_{n\rightarrow\infty}r_{0}(\Sigma_{n})=+\infty
\]
we have 
\[
\lim_{n\rightarrow\infty}r_{0}(\Sigma_{n})/r_{1}(\Sigma_{n})=1.
\]

\end{thm}

\begin{rmk}The condition that ``the spheres separate the compact
part from infinity'' is necessary as S.Brendle and M.Eichmair proved
a non unique result of CMC sphere without this condition, see \cite{OUTLYING-CMC}.
However, it is still unknown (except in Qing and Tian's case) that
whether a sphere that separates a sufficiently large compact set from
infinity with a particular constant mean curvature is unique . The
stability condition is necessary from technical point of view. In
\cite{NERZ-CMC}, Nerz did not use stability condition. However, he
used the following condition
\[
\int_{\Sigma}H^{2}(\Sigma)d\mu-16\pi(1-\dot{g})\leq\frac{C_{1}^{*}}{(r')^{\eta}}
\]
where $\dot{g}$ is the genus of the surfaces. In the case of spheres,
$\dot{g}=0.$ This condition is similar to the conclusion of Lemma
\ref{integral estimate} in this paper which is the consequence of
stability. Actually, this is the only place where we use stability
condition.

\end{rmk}

We can use Theorem \ref{thm1} to remove the radius condition in Huang's
work \cite{Huang-CMC} in the case of decay rate $q=1$. In Huang's
paper, she studied the following kind of asymptotically flat initial
data set.

\begin{Def}

A three-manifold $M$ with a Riemannian metric $g$ and a two-tensor
$K$ is called an initial data set (IDS) if $g$ and $K$ satisfy
the constraint equations

\begin{eqnarray}
R_{g}-|K|_{g}^{2}+(tr_{g}(K))^{2} & = & 16\pi\rho\nonumber \\
div_{g}(K)-d(tr_{g}(K)) & = & 8\pi J\label{eq:Constraint equations}
\end{eqnarray}
 where $R_{g}$ is the scalar curvature of the metric $g$, $tr_{g}(K)$
denotes $g^{ij}K_{ij}$, $\rho$ is the observed energy density, and
$J$ is the observed momentum density. 

\end{Def}

\begin{Def} We say $(M,g,K)$ is asymptotically flat initial data
set (AF-IDS) at the decay rate $q\in(\frac{1}{2},1]$ if it is an
initial data set, and there is a compact subset $\widetilde{K}\subset M$
such that $M\setminus\widetilde{K}$ is diffeomorphic to $\mathbb{R}^{3}\setminus B_{1}(0)$
and there exists coordinate $\{x^{i}\}$ with respect to which the
metric $g$ can be written as

\[
g_{ij}(x)=\delta_{ij}+h_{ij}(x)
\]

\begin{eqnarray*}
h_{ij}(x)=O_{5}(|x|^{-q}) &  & K_{ij}(x)=O_{1}(|x|^{-1-q})
\end{eqnarray*}
 Also, $\rho$ and $J$ satisfy

\begin{eqnarray*}
\rho(x)=O(|x|^{-2-2q}) &  & J(x)=O(|x|^{-2-2q})
\end{eqnarray*}

Here, $f=O_{k}(|x|^{-q})$ means $\partial^{l}f=O(|x|^{-l-q})$ for
$l=0,\cdots,k$. \end{Def} 

The RT conditions used by Huang is the following. 

\begin{Def} We say $(M,g,K)$ is AF-RT-IDS if it is AF-IDS, and g,
K satisfy 
\begin{eqnarray*}
h_{ij}^{odd}(x)=O_{2}(|x|^{-1-q}) &  & K_{ij}^{even}(x)=O_{1}(|x|^{-2-q})
\end{eqnarray*}
 Also, $\rho$ and $J$ satisfy

\begin{eqnarray*}
\rho^{odd}(x)=O(|x|^{-3-2q}) &  & J^{odd}(x)=O(|x|^{-3-2q})
\end{eqnarray*}
 where $f^{odd}(x)=f(x)-f(-x)$ and $f^{even}(x)=f(x)+f(-x).$\end{Def}

For this kind of AF-RT-IDS , the center of mass $C$ is defined as
\begin{equation}
C^{k}=\frac{1}{16\pi m}\lim_{r\rightarrow\infty}(\int_{|x|=r}x^{k}(h_{ij,i}-h_{ii,j})v^{j}d\mu-\int_{|x|=r}(h_{ik}v^{i}-h_{ii}v^{k})d\mu).
\end{equation}
From \cite{Huang-center-of-mass}, we know it is well defined.

In \cite{Huang-CMC}, Huang proved that for AF-RT manifold $(M,g,K)$
at the decay rate $q\in(\frac{1}{2},1]$ with non zero mass, there
exists a foliation of CMC spheres in the exterior region of the manifold.
Moreover if the mass is positive each CMC surface is strictly stable.
Under some radius condition Huang also proved the uniqueness of the
foliation. Namely

\paragraph{Huang's uniqueness theorem:}

Assume $(M,g,K)$ is AF-RT-IDS at the decay rate $q\in(\frac{1}{2},1]$
and $m>0$. Then there exists $ $some $\sigma_{1}$ so that if $\Sigma$
has the following properties,
\begin{enumerate}
\item $\Sigma$ is topologically a sphere,
\item $\Sigma$ has constant mean curvature $H=H_{\Sigma_{R}}$ for some
$R\geq\sigma_{1},$
\item $\Sigma$ is stable,
\item $r_{0}\geq H(\Sigma)^{-a}$ or $r_{1}\leq C_{1}r_{0}^{1/a}$ for some
$a$ satisfying $\frac{5-q}{2(2+q)}<a\leq1,$
\end{enumerate}
then $\Sigma$ is one of the CMC surface constructed by Huang's existence
result.$\square$

The most interesting case in general relativity is when the decay
rate $q=1.$ In this case we know the condition of Huang means $r_{1}\leq C_{1}r_{0}^{1/a}$
for $2/3<a\leq1.$ This is the same constraint condition as indicated
in \cite{Metzger-prescribe-mean-curvature}. From Theorem \ref{thm1}
we can remove the radius condition in the case of $q=1.$ We have

\begin{cor}\label{remove Huang's radius}

Suppose $(M,g,K)$ is an AF-RT-IDS manifold at the decay rate $1$
with positive mass. Then there exists a compact set $\tilde{K}$,
such that for any $H>0$ sufficiently small, there is only one stable
sphere with constant mean curvature $H$ that separates infinity from
$\tilde{K}.$

\end{cor}

This corollary follows from Theorem \ref{thm1} and Huang's uniqueness
theorem.

Nerz studied the existence and uniqueness of CMC foliation in $C_{\frac{1}{2}+\varepsilon}^{2}$-AF
manifolds without RT conditions in \cite{NERZ-CMC}. See Theorem 3.1
Theorem 3.2 and Theorem 3.3 in \cite{NERZ-CMC} for details of these
results. The method in this paper also applies to certain cases in
his setting. Namely, in $C_{1}^{4}$-AF manifolds, we can improve
the uniqueness result of Nerz.

\paragraph{Nerz's Theorem}

Let $(M,g)$ be $C_{\frac{1}{2}+\varepsilon}^{2}$-AF manifold with
non 0 mass. Then there is a constant $\sigma_{0}$, a compact set
$\hat{K}\subset M$ and a diffeomorphism 
\[
\bar{\Phi}:(\sigma_{0},+\infty)\times S^{2}\rightarrow M\backslash\hat{K}
\]
such that each sphere 
\[
\Sigma_{\sigma}=\bar{\Phi}(\sigma\times S^{2})
\]
 has constant mean curvature $H=\frac{2}{\sigma}$. Each sphere is
stable if and only if $m>0.$ 

Moreover, when $m\neq0$, all hypersurfaces $\Sigma\subset\mathscr{A}_{r'}^{\varepsilon,\eta}(C_{0}^{*},C_{1}^{*}),C_{0}^{*}\in[0,1)$
with constant mean curvature $H=\frac{2}{\sigma},\sigma\geq\sigma_{0}$
coincide. Here $\mathscr{A}_{r'}^{\varepsilon,\eta}(C_{0}^{*},C_{1}^{*})$
means the set of closed oriented genus $\dot{g}$ hypersurfaces $\Sigma$
with 
\begin{equation}
|\vec{z}|\leq C_{0}^{*}r'+C_{1}^{*}r'^{1-\eta},r'^{4+\eta}\leq\min_{\Sigma}|\bar{x}|^{5+2\varepsilon},\int_{\Sigma}H^{2}(\Sigma)d\mu-16\pi(1-\dot{g})\leq\frac{C_{1}^{*}}{(r')^{\eta}}\label{Nerz's three conditions}
\end{equation}
where $r'=\sqrt{\frac{|\Sigma|}{4\pi}}$ and $\vec{z}=\fint_{\Sigma}x_{i}d\mu_{e}$
is the Euclidean center of the surface $\Sigma.$ $\square$

If $\Sigma$ is a standard sphere in $\mathbb{R}^{3}$, $r'$ is the
radius of the sphere. In the case of CMC sphere in AF manifolds, this
is also true roughly. So in this case 
\[
|\vec{z}|\leq C_{0}^{*}r'+C_{1}^{*}r'^{1-\eta},C_{0}^{*}\in[0,1)
\]
 means the outer radius $r_{1}$ and inner radius $r_{0}$ satisfies
$r_{1}\leq Cr_{0}.$ 

From Theorem \ref{thm1}, we can get

\begin{cor}\label{remove Nerz's radius} Let $(M,g)$ be $C_{1,\tau}^{4}$-AF
manifold, with $m>0$, then there exists a compact set $\tilde{K}$
such that any stable sphere that separates $\tilde{K}$ from infinity
with constant mean curvature belongs to $\{\Sigma_{\sigma}|\sigma\geq\sigma_{0}\}$
constructed by Nerz.

\end{cor}

Now we sketch the proof of the paper and state the main contributions
of us. In the proof of the uniqueness, one central step is to calculate
the following integral on the stable CMC sphere
\begin{equation}
\int_{\Sigma}(H-H_{e})<v_{e},b>d\mu_{e}\label{central integral}
\end{equation}
where $H$ and $H_{e}$ are mean curvature in the metric $g$ and
the Euclidean metric and $v_{e}$ is unit normal vector in Euclidean
metric and $b$ is a constant vector to be chosen. See \cite{Huisken-Yau,Qing-Tian-CMC,Shiguang-Ma-CMC}
for different methods to calculate this integral. When the radius
pinching estimate $r_{1}\leq Cr_{0}$ fails to hold, one want to relate
this integral to the mass. To calculate this integral, first one need
do a priori estimates for $\Sigma,$ so that we have a good domain.
We carry out Qing and Tian's method in the general metrics. In Section
2 and Section 3, we do curvature estimates and blow down analysis.
We only require the manifolds to be $C_{q,\tau}^{4}$-AF. In Section
$4$, we use harmonic map techniques to analyze the intermediate part
where we start to require the metric to be $C_{1,\tau}^{4}$-AF. We
improve Qing and Tian's estimates on the second fundamental form in
Lemma \ref{The new estimate for the second fundamental form}. The
most hard point is to analyze the expression of (\ref{central integral}).
In the case of asymptotically Schwarzschild manifolds, one can reduce
(\ref{central integral}) to explicit form to calculate, see \cite{Qing-Tian-CMC}.
I used harmonic coordinates to reduce (\ref{central integral}) to
explicit form in \cite{Shiguang-Ma-CMC}. Actually, the metric needs
to be close to Schwarzschild in certain sense when doing so. When
the metric is only $C_{1,\tau}^{4}$-AF, it is impossible to reduce
it to explicit form. So one needs to find a general way to calculate
it. In Section 5, we found such a direct way and manage to relate
it to the mass when $r_{1}\leq Cr_{0}$ fails to hold. This needs
much calculation in geometry as well as analysis. RT conditions or
radius conditions of any type are not needed. It seems that this direct
method is the most natural way to calculate this integral. At last
we can prove Corollary \ref{remove Huang's radius} and Corollary
\ref{remove Nerz's radius} easily from Theorem \ref{thm1}.
\begin{description}
\item [{Acknowledgement}] I would like to thank my advisor Professor Gang
Tian for long time help and encouragement. I show my special thanks
to Professor Jie Qing for helpful discussions. I also thank Yalong
Shi for discussions on harmonic maps.
\end{description}

\section{Curvature estimates}

In this section and the next section we assume $(M,g)$ is $C_{q,\tau}^{4}$-AF
manifold at the decay rate $1/2<q\leq1$ and $\tau>0$. Let $\Sigma$
be a constant mean curvature sphere which separates the compact part
$K'$ with infinity . First as a small generalization of Lemma 5.2
in \cite{Huisken-Yau}, we have

\begin{lem} Let $X=x^{i}\frac{\partial}{\partial x^{i}}$ be the
Euclidean coordinate vector field and $r=(\Sigma(x^{i})^{2})^{1/2}$
. Then we have the estimate:
\[
\int_{\Sigma}<X,v>^{2}r^{-4}d\mu\leq H^{2}|\Sigma|.
\]
Moreover for each $a\geq a_{0}>2$ and $r_{0}$ sufficiently large
, we have:
\[
\int_{\Sigma}r^{-a}d\mu\leq C(a_{0})r_{0}^{2-a}H^{2}|\Sigma|
\]

\end{lem}

\begin{proof} Because the mean curvature $H$ is constant, then for
a smooth vector field $Y$ on $\Sigma$, we have the divergence formula:
\[
\int_{\Sigma}{\rm div}_{\Sigma}Yd\mu=H\int_{\Sigma}<Y,v>d\mu.
\]
 We choose $Y=Xr^{-a}$, $a\geq2$$ $ and $e_{\alpha}(p)$ is the
orthonormal basis of $T_{p}\Sigma$, $\alpha=1,2$. Suppose $e_{\alpha}=a_{\alpha}^{i}\frac{\partial}{\partial x^{i}}$,
it is obvious that $a_{\alpha}^{i}$ is bounded because the manifold
is asymptotically flat. Then we have: 
\begin{align*}
{\rm div_{\Sigma}}Y & ={\rm div_{\Sigma}}(Xr^{-a})=<\nabla_{e_{\alpha}}(Xr^{-a}),e_{\alpha}>\\
 & =r^{-a}{\rm div_{\Sigma}}X-ar^{-a-2}a_{\alpha}^{i}a_{\alpha}^{j}x^{i}x^{j}+O(r^{-a-q})\\
 & =r^{-a}{\rm div_{\Sigma}}X-ar^{-a-2}|X^{\tau}|^{2}+O(r^{-a-q}),
\end{align*}
where $X^{\tau}$ is the tangential projection of $X$.
\[
|{\rm div_{\Sigma}}X-2|=O(r^{-q}),
\]

Note that $|X^{\tau}|^{2}=r^{2}-<X,v>^{2}+O(r^{2-q})$, then combine
all of these we have: 
\begin{equation}
|(2-a)\int_{\Sigma}r^{-a}d\mu+a\int_{\Sigma}<X,v>^{2}r^{-a-2}d\mu-H\int_{\Sigma}<X,v>r^{-a}d\mu|\leq C\int_{\Sigma}r^{-a-q}d\mu\label{eq:1}
\end{equation}

Choosing $a=2$, from Holder inequality, we have:
\[
\int_{\Sigma}<X,v>^{2}r^{-4}d\mu\leq\frac{1}{4}H^{2}|\Sigma|+C\int_{\Sigma}r^{-2-q}d\mu.
\]
Then we choose $a=2+q,$
\[
\int_{\Sigma}r^{-2-q}d\mu\leq4r_{0}^{-q}(\int_{\Sigma}<X,v>^{2}r^{-4}d\mu+H^{2}|\Sigma|+C\int_{\Sigma}r^{-2-q}d\mu).
\]
Then from the two inequalities above, we can deduce:

\[
\int_{\Sigma}<X,v>^{2}r^{-4}d\mu\leq H^{2}|\Sigma|.
\]
Now from \ref{eq:1} we get 
\[
\int_{\Sigma}r^{-a}d\mu\leq C(a_{0}-2)^{-1}r_{0}^{2-a}H^{2}|\Sigma|.
\]

\end{proof}

Due to \cite{Huisken-Yau} Proposition 5.3, we can deduce integral
estimate for $|\AA|$ from the stability property.

\begin{lem}\label{integral estimate}

Suppose $\Sigma$ is a stable constant mean curvature sphere in the
$C_{q,\tau}^{4}$-asymptotically flat manifold. For $r_{0}$ sufficiently
large$\AA$, we have 
\begin{align}
\int_{\Sigma}|\AA|^{2}d\mu & \leq Cr_{0}^{-q}\label{eq:integral estimate for the second fundamental form}\\
\int_{\Sigma}H^{2}d\mu & =16\pi+O(r_{0}^{-q})\label{integral estimate of H}
\end{align}

\end{lem} 

\begin{proof}We use the stability of $\Sigma.$ For any $f$ which
satisfies
\[
\int_{\Sigma}fd\mu=0,
\]
we have 
\[
\int_{\Sigma}|\nabla f|^{2}d\mu\geq\int_{\Sigma}(|A|^{2}+Ric(v,v))f^{2}d\mu
\]
where $A$ is the second fundamental form of $\Sigma$ and $Ric$
is the Ricci curvature of $M$.

Choose $\psi$ to be a conformal map of degree 1 from $\Sigma$ to
the standard $S^{2}\subset\mathbb{R}^{3}$ . Each component $\psi_{i}$
of $\psi$ can be chosen such that $\int\psi_{i}d\mu=0$ , see \cite{P.Li-Yau-Conformal-invariant}
. We have for each $\psi_{i}$ 
\[
\int_{\Sigma}|\nabla\psi_{i}|^{2}d\mu=\frac{8\pi}{3}.
\]

Since $\ensuremath{\sum\psi_{i}^{2}\equiv1},$ we have
\[
\int_{\Sigma}(|A|^{2}+Ric(v,v))d\mu\leq8\pi.
\]
From Gauss equation

\[
\frac{1}{2}|A|^{2}+Ric(v,v)-\frac{1}{2}R+K=\frac{1}{2}H^{2}
\]
we have:

\[
|A|^{2}+Ric(v,v)=\frac{1}{2}|\AA|^{2}+\frac{3}{4}H^{2}+\frac{1}{2}R-K
\]
where $K$ is the Gauss curvature of $\Sigma$ and $\AA_{ij}=A_{ij}-\frac{H}{2}g_{ij}$.
Then
\[
\int_{\Sigma}\frac{1}{2}|\AA|^{2}+\frac{3}{4}H^{2}|\Sigma|\leq12\pi+Cr_{0}^{-q}H^{2}|\Sigma|
\]
 because $R=O(r^{-3-\tau})$. So we have 
\[
\int_{\Sigma}H^{2}d\mu=\ensuremath{H^{2}|\Sigma|\leq16\pi}+O(r_{0}^{-q}).
\]
 Using the Gauss equation in another way, we have
\begin{align*}
 & \int_{\Sigma}|\AA|^{2}d\mu\\
= & \int_{\Sigma}(|A|^{2}-\frac{H^{2}}{2})d\mu\\
= & \frac{1}{2}\int_{\Sigma}(|A|^{2}+Ric(v,v))d\mu+\frac{1}{2}\int_{\Sigma}(R-3Ric(v,v)-2K)d\mu\\
\leq & \int_{\Sigma}r^{-2-q}d\mu=O(r_{0}^{-q}).
\end{align*}

\end{proof}

\begin{lem}\label{Willmore functional} Suppose $\Sigma$ is a CMC
sphere in an asymptotically flat end $(R^{3}\setminus B_{1}(0))$,
then we have:

\[
\int_{\Sigma}H_{e}^{2}d\mu_{e}=16\pi+O(r_{0}^{-q}),
\]
 where $H_{e}$ denotes the mean curvature with respect to the background
Euclidean metric.

\end{lem}

\begin{proof} First we follow the calculation of Huisken and Ilmanen,
see\cite{Huisken-Ilmanen-IMCF}.Now 
\[
g_{ij}=\delta_{ij}+h_{ij}
\]

Suppose 
\[
g_{ij}|_{\Sigma}=f_{ij}~~~~\delta_{ij}|_{\Sigma}=\varepsilon_{ij}
\]
where $f^{ij}$ and $\varepsilon^{ij}$ are the corresponding inverse
matrices. $v,\omega,A,H,d\mu$ represent the normal vector , the dual
form of $v$, the second fundamental form , the mean curvature and
the volume form of $\Sigma$ in the metric $g$. And $v_{e},\omega_{e},A_{e},H_{e},\mu_{e}$
represent the corresponding ones in Euclidean metric. Through easy
calculation, we have 
\begin{align}
f^{ij}-\varepsilon^{ij} & =-f^{ik}h_{kl}f^{lj}\pm C|h|^{2}\label{eq:f^ij-e^ij}\\
g^{ij}-\delta^{ij} & =-g^{ik}h_{kl}g^{lj}\pm C|h|^{2}\label{eq:g^ij-delta^ij}
\end{align}
 
\begin{align}
\omega & =\frac{\omega_{e}}{|\omega_{e}|}\nonumber \\
v^{i} & =g^{ij}\omega_{j}\nonumber \\
\ensuremath{(\omega_{e})_{i}} & =\omega_{i}\pm C|P|\label{w_e-w}\\
v_{e}^{i} & =v^{i}+C|h|\label{v_e-v}\\
1-|\omega_{e}| & =\frac{1}{2}h_{ij}v^{i}v^{j}\label{1-w_e}
\end{align}
\begin{equation}
\Gamma_{ij}^{k}=\frac{1}{2}g^{kl}(\overline{\nabla}_{i}h_{jl}+\overline{\nabla}_{j}h_{il}-\overline{\nabla}_{l}h_{ij})\pm C|h||\overline{\nabla}h|\label{Gamma_{ijk}}
\end{equation}
and $\Gamma_{ij}^{k}$ is the Christoffel symbol for $\overline{\nabla}-\overline{\nabla}_{e}$,
where we denote the gradient operator in the metric $g$ and $\delta$
by $\overline{\nabla}$ and $\overline{\nabla}_{e}$ and because the
metric $g$ and $\delta$ are uniformly equivalent , we have:

\[
C^{-1}d\mu\leq d\mu_{e}\leq Cd\mu
\]

We have the formula: 
\begin{equation}
|\omega_{e}|_{g}A_{ij}=(A_{e})_{ij}-(\omega_{e})_{k}\Gamma_{ij}^{k}\label{key expression for second fundamental form}
\end{equation}

So we have

\begin{align*}
H-H_{e} & =f^{ij}A_{ij}-\varepsilon^{ij}(A_{e})_{ij}\\
 & =(f^{ij}-\varepsilon^{ij})A_{ij}+\varepsilon^{ij}A_{ij}(1-|\omega_{e}|_{g})+\varepsilon^{ij}(|\omega_{e}|_{g}A_{ij}-(A_{e})_{ij})
\end{align*}
 From (\ref{eq:f^ij-e^ij},\ref{eq:g^ij-delta^ij},\ref{w_e-w},\ref{v_e-v},\ref{1-w_e}),
we have 
\[
\varepsilon^{ij}A_{ij}(1-|\omega_{e}|_{g})=\frac{1}{2}Hv^{i}v^{j}h_{ij}\pm C|h|^{2}|A|
\]
 and using (\ref{eq:f^ij-e^ij},\ref{eq:g^ij-delta^ij},\ref{w_e-w},\ref{v_e-v},\ref{1-w_e},\ref{Gamma_{ijk}},\ref{key expression for second fundamental form})
we have:

\begin{align*}
 & \varepsilon^{ij}(|\omega_{e}|A_{ij}-(A_{e})_{ij})\\
= & -\varepsilon^{ij}(\omega_{e})_{k}\Gamma_{ij}^{k}\\
= & -\frac{1}{2}f^{ij}\omega_{k}g^{kl}(\overline{\nabla}_{i}h_{jl}+\overline{\nabla}_{j}h_{il}-\overline{\nabla}_{l}h_{ij})\pm C|h||\overline{\nabla}h|\\
= & -f^{ij}v^{l}\overline{\nabla}_{i}h_{jl}+\frac{1}{2}f^{ij}v^{l}\overline{\nabla}_{l}h_{ij}\pm C|h||\overline{\nabla}h|
\end{align*}

At last , we have 

\begin{eqnarray}
 &  & H-H_{e}=-f^{ik}h_{kl}f^{lj}A_{ij}+\frac{1}{2}Hv^{i}v^{j}h_{ij}-f^{ij}v^{l}\overline{\nabla}_{i}h_{jl}\label{H-H_e}\\
 &  & +\frac{1}{2}f^{ij}v^{l}\overline{\nabla}_{l}h_{ij}\pm C|h||\overline{\nabla}h|\pm C|h|^{2}|A|
\end{eqnarray}
\begin{align*}
\int_{\Sigma}H_{e}^{2}d\mu_{e} & =(1+O(r_{0}^{-q}))\int_{\Sigma}H_{e}^{2}d\mu\\
 & \leq(1+O(r_{0}^{-q}))(\int_{\Sigma}H^{2}d\mu+\int_{\Sigma}(H_{e}-H)^{2}+2|H(H_{e}-H)|d\mu)\\
 & \leq(1+O(r_{0}^{-q}))(16\pi+O(r_{0}^{-q})+\int_{\Sigma}(H_{e}-H)^{2}\\
 & +2(\int_{\Sigma}H^{2}d\mu)^{\frac{1}{2}}(\int_{\Sigma}(H_{e}-H)^{2}d\mu)^{\frac{1}{2}})
\end{align*}

\begin{align*}
\int(H_{e}-H)^{2}d\mu & \leq\int O(|x|^{-2q})|A|^{2}+H^{2}O(|x|^{-2q})+O(|x|^{-2-2q})d\mu\\
 & \leq\int O(|x|^{-2q})H^{2}+O(|x|^{-2q})|{\AA}|^{2}+O(|x|^{-2-2q})d\mu\\
 & =O(r_{0}^{-2q})
\end{align*}
so we have

\[
\int_{\Sigma}H_{e}^{2}d\mu_{e}\leq16\pi+O(r_{0}^{-q})
\]

On the other hand, by Euler formula, 
\[
K_{e}=\frac{1}{4}H_{e}^{2}-\frac{1}{2}|\AA_{e}|^{2}.
\]

So we have 
\[
\int H_{e}^{2}d\mu_{e}\geq16\pi
\]
 which implies: 
\[
\int_{\Sigma}H_{e}^{2}d\mu_{e}=16\pi+O(r_{0}^{-q}).
\]

\end{proof}

Based on Michael and Simon \cite{Michael-Simon-Sobolev}, we have
the following Sobolev inequality. 

\begin{lem}\label{Sobolev} Suppose $\Sigma$ is a CMC sphere in
the asymptotically flat end with $r_{0}$ sufficiently large and that
$\int_{\Sigma}H^{2}\leq C,$ then:
\begin{equation}
(\int_{\Sigma}f^{2}d\mu)^{1/2}\leq C(\int_{\Sigma}|\nabla f|d\mu+\int_{\Sigma}H|f|d\mu).\label{eq:Sobolev inequality}
\end{equation}

\end{lem}

\begin{proof}First this is valid for the surface in Euclidean Space.
So by the uniform equivalence of the metric $g$ and $\delta$ , we
have: 
\begin{align*}
 & (\int|f|^{2}d\mu)^{\frac{1}{2}}\\
\leq & C(\int|f|^{2}d\mu_{e})^{\frac{1}{2}}\\
\leq & C(\int|\nabla f|+H|f|+|H-H_{e}||f|d\mu)
\end{align*}
To bound the last term on the right , we notice:
\begin{align*}
 & \int|H-H_{e}||f|d\mu\\
\leq & \int O(|x|^{-q})|A||f|+O(|x|^{-q})H|f|+O(|x|^{-1-q})|f|d\mu\\
\leq & O(r_{0}^{-q})\int H|f|+(\int|\AA|^{2}d\mu)^{\frac{1}{2}}O(r_{0}^{-q})\|f\|_{L^{2}}+O(r_{0}^{-q})\|f\|_{L^{2}}
\end{align*}
So we can combine the two inequalities and choose $r_{0}$ sufficiently
large and get the desired result. 

\end{proof}

\begin{lem}\label{The relationship of H and r1} Suppose $\Sigma$
is a CMC sphere in an asymptotically flat end with $r_{0}(\Sigma)$
sufficiently large, then:

\[
C_{1}H^{-1}\leq diam(\Sigma)\leq C_{2}H^{-1},
\]
where the $diam(\Sigma)$ denotes the diameter of $\Sigma$ in the
Euclidean space $\mathbb{R}^{3}.$

In particular, if the surface $\Sigma$ separates the infinity from
the compact part $K'$, then:

\[
C_{1}H^{-1}\leq r_{1}(\Sigma)\leq C_{2}H^{-1}.
\]

\end{lem}

\begin{proof}We already know that:

\[
\int_{\Sigma}H_{e}^{2}d\mu_{e}=16\pi+O(r_{0}^{-q})
\]

Then from \cite{L.Simon-Willmore} Lemma 1.1, we know

\[
\sqrt{\frac{2|\Sigma|_{e}}{F(\Sigma)}}\leq diam(\Sigma)\leq C\sqrt{|\Sigma|_{e}F(\Sigma)}
\]
where $F(\Sigma)=\frac{1}{2}\int_{\Sigma}H_{e}^{2}d\mu_{e}$ is the
Willmore functional and $|\Sigma|_{e}$ is the volume of $\Sigma$
with respect to the Euclidean metric. But the Euclidean metric is
uniformly equivalent to $g$. From 
\[
\int_{\Sigma}H^{2}d\mu=16\pi+O(r^{-q}),
\]
we know $\tilde{C}_{1}H^{-1}\leq|\Sigma|_{e}\leq\tilde{C}_{2}H^{-1}$
for some $\tilde{C}_{1},\tilde{C}_{2}>0.$ So we get the result. 

\end{proof}

To get the pointwise estimate for $\AA$, we use the Simons identity

\begin{lem}\cite{Schoen-Simon-Yau-minimal-surface} Suppose $\Sigma$
is a hypersurface in a Riemannian manifold $(M,g)$. Then the second
fundamental form satisfies the following identity:

\begin{align*}
\Delta A_{ij}= & \nabla_{i}\nabla_{j}H+HA_{ik}A_{jk}-|A|^{2}A_{ij}+HR_{3i3j}-A_{ij}R_{3k3k}+A_{jk}R_{klil}\\
 & +A_{ik}R_{kljl}-2A_{lk}R_{iljk}+\overline{\nabla}_{j}R_{3kik}+\overline{\nabla}_{k}R_{3ijk}
\end{align*}
where $R_{ijkl}$ and $\overline{\nabla}$ are the curvature and gradient
operator of $(M,g)$. Then for CMC surfaces we can deduce the following
inequality:

\begin{align*}
-|\AA|\Delta|\AA|\leq & |\AA|^{4}+CH|\AA|^{3}+CH^{2}|\AA|^{2}+C|\AA|^{2}|x|^{-2-q}\\
 & +CH|\AA||x|^{-2-q}+C|\AA||x|^{-3-q}
\end{align*}

We also need an inequality for $\nabla\AA$ because we also want to
estimate the higher derivative: 

\begin{align*}
-|\nabla\AA|\Delta|\nabla\AA|\leq & C|\nabla\AA|^{2}(|\AA|^{2}+H|\AA|+H^{2}+O(|x|^{-2-q}))\\
 & +|\nabla\AA|((|\AA|^{2}+H|\AA|+H^{2})O(|x|^{-2-q})+(|\AA|+H)O(|x|^{-3-q})+O(|x|^{-4-q}))
\end{align*}

\end{lem}

Then from Simons identity and Sobolev inequality Lemma \ref{Sobolev}
we have the following basic curvature estimate:

\begin{thm}\label{Main curvature estimates}\cite{Qing-Tian-CMC}
Suppose that $(\mathbb{R}^{3}\setminus B_{1}(0),g)$ is an asymptotically
flat end. Then there exist positive numbers $\sigma_{0}$, $\delta_{0}$
such that for any CMC sphere in the end, which separates the infinity
from the compact part, we have:

\[
|\AA|^{2}(x)\leq C|x|^{-2}\int_{B_{\delta_{0}|x|}(x)}|\AA|^{2}d\mu+C|x|^{-2-2q}\leq C|x|^{-2}r_{0}^{-q}
\]

\[
|\nabla\AA|^{2}(x)\leq C|x|^{-2}\int_{B_{\delta_{0}|x|}(x)}|\nabla\AA|^{2}d\mu+C|x|^{-4-2q}\leq C|x|^{-4}r_{0}^{-q/2}
\]
provided that $r_{0}\geq\sigma_{0}$.

\end{thm}

\begin{proof}In the Sobolev inequality (\ref{eq:Sobolev inequality})
we take $f=u^{2},$ then we have

\begin{align*}
(\int_{\Sigma}u^{4}d\mu)^{\frac{1}{2}} & \leq C(2\int_{\Sigma}|u||\nabla u|d\mu+\int_{\Sigma}Hu^{2}d\mu)\\
 & \leq C(\int_{\Sigma}u^{2})^{\frac{1}{2}}(\int_{\Sigma}|\nabla u|^{2}d\mu)^{\frac{1}{2}}+C(\int_{supp(u)}H^{2}d\mu)^{\frac{1}{2}}(\int_{\Sigma}u^{4}d\mu)^{\frac{1}{2}}.
\end{align*}
 We need the following lemma

\begin{lem}For any $\varepsilon>0$, we can find a uniform $\delta_{0}$
sufficiently small such that for any $x\in\Sigma$ , we have: 

\[
\int_{B_{\delta_{0}|x|}(x)}H^{2}d\mu\leq\varepsilon
\]

\end{lem}

\begin{proof}The metric $g$ is uniform equivalent to Euclidean metric
$\delta$. So we only need to prove that there exist $C$, such that:
\[
|B_{\delta_{0}|x|}(x)|_{e}\leq C\delta_{0}^{2}|x|^{2}.
\]
Then we have 
\[
H^{2}|B_{\delta_{0}|x|}(x)|_{e}\leq C\delta_{0}^{2}|x|^{2}H^{2}\leq C\delta_{0}^{2}.
\]
From the proof of Lemma 1.1 in \cite{L.Simon-Willmore}, we know,
for any $x\in\Sigma$, $B_{\sigma}(x)$ denotes the Euclidean ball
of radius $\sigma$ with center $x$ in $\mathbb{R}^{3}$, $\Sigma_{\sigma}=\Sigma\cap B_{\sigma}(x)$,
then there exists $C$ such that for $0<\sigma\leq\rho<\infty$

\[
\sigma^{-2}|\Sigma_{\sigma}|_{e}\leq C(\rho^{-2}|\Sigma_{\rho}|_{e}+F(\Sigma_{\rho}))
\]
where $F(\Sigma_{\rho})$ is the Willmore functional. $C$ doesn't
depend on $\Sigma,\sigma,\rho$. Let $\rho\rightarrow\infty$ , $\rho^{-2}|\Sigma_{\rho}|\rightarrow0$,
so we have: 
\[
\sigma^{-2}|\Sigma_{\sigma}|_{e}\leq CF(\Sigma)\leq C.
\]
So we proved this lemma.

\end{proof}

So if $supp(u)\subset B_{\delta_{0}|x|}(x)$, we have the following
scaling invariant Sobolev inequality:
\[
(\int_{\Sigma}u^{4}d\mu)^{\frac{1}{2}}\leq C(\int_{\Sigma}u^{2})^{\frac{1}{2}}(\int_{\Sigma}|\nabla u|^{2}d\mu)^{\frac{1}{2}}
\]

\begin{lem}\cite{Qing-Tian-CMC} Suppose that a nonnegative function
$v\in L^{2}$ solves
\[
-\Delta v\leq fv+h
\]
on $B_{2R}(x_{0})$, where
\[
\int_{B_{2R}(x_{0})}f^{2}d\mu\leq CR^{-2}
\]
and $\ensuremath{h\in L^{2}(B_{2R}(x_{0}))}.$ And suppose that 
\[
(\int_{\Sigma}u^{4}d\mu)^{\frac{1}{2}}\leq C(\int_{\Sigma}u^{2})^{\frac{1}{2}}(\int_{\Sigma}|\nabla u|^{2}d\mu)^{\frac{1}{2}}
\]
holds for all $u$ with support inside $B_{2R}(x_{0})$. Then 
\[
\sup_{B_{R}(x_{0})}v\leq CR^{-1}\|v\|_{L^{2}(B_{2R}(x_{0}))}+CR\|h\|_{L^{2}(B_{2R(x_{0})})}.
\]

\end{lem}The proof of this lemma is Moser's iteration. See \cite{Qing-Tian-CMC}
Lemma 2.6 for the proof.

We have
\begin{align*}
-\Delta|\AA| & \leq(|\AA|^{2}+H^{2}+H|\AA|+C|x|^{-2-q})|\AA|+CH|x|^{-2-q}+C|x|^{-3-q}\\
 & =f_{1}|\AA|+h_{1}
\end{align*}
\begin{align*}
-\Delta|\nabla\AA|\leq & C|\nabla\AA|(|\AA|^{2}+H|\AA|+H^{2}+O(|x|^{-2-q}))\\
 & +((|\AA|^{2}+H|\AA|+H^{2})O(|x|^{-2-q})+(|\AA|+H)O(|x|^{-3-q})+O(|x|^{-4-q}))\\
= & f_{2}|\nabla\AA|+h_{2}.
\end{align*}
We need to prove that $\|f_{1}\|_{L^{2}(B_{2\delta_{0}|x|}(x))}^{2},\|f_{2}\|_{L^{2}(B_{2\delta_{0}|x|}(x))}^{2}\leq C|x|^{-2}$
, 

Choose a proper cut off function on $\Sigma$, by using the Simons
identity and inequality we can get 
\[
\int_{\Sigma\cap B_{2\delta_{0}|x_{0}|}(x_{0})}|\AA|^{4}d\mu\leq C|x_{0}|^{-2}\int_{\Sigma}|\AA|^{2}d\mu
\]
From this, we can get the estimate for $f_{1}$ and $f_{2}$. And
it is easy to show that $\|h_{1}\|_{L^{2}(B_{2\delta_{0}|x|}(x))}^{2}=O(|x|^{-4-2q})$
and $\|h_{2}\|_{L^{2}(B_{2\delta_{0}|x|}(x))}^{2}=O(|x|^{-6-2q})$
in the same way. 

Now we know: 
\[
\int_{B_{\delta_{0}|x|}(x)}|\AA|^{2}d\mu\leq Cr_{0}^{-q}
\]
 and 
\[
\int_{B_{\delta_{0}|x|}(x)}|\nabla\AA|^{2}d\mu\leq|x|^{-2}(\int_{B_{\delta_{0}|x|}(x)}|\AA|^{2}d\mu)^{\frac{1}{2}}\leq|x|^{-2}r_{0}^{-\frac{q}{2}}.
\]
 The first inequality follows from (\ref{eq:integral estimate for the second fundamental form}).
The second one follows from the first one and Simon's identity. Finally
we proved Theorem \ref{Main curvature estimates}.

\end{proof}

\begin{rmk}From Theorem \ref{Main curvature estimates} and (\ref{key expression for second fundamental form})
and the differentiation of (\ref{key expression for second fundamental form})
and Lemma \ref{The relationship of H and r1}, we have 

\begin{align}
|A_{e}| & \leq C|x|^{-1}r_{0}^{-\frac{q}{2}}\label{estimate of second fundamental in Euclidean metric}\\
|\nabla_{e}A_{e}| & \leq C|x|^{-2}r_{0}^{-\frac{q}{4}}\label{estimates of derivative of second fundamental in Euclidean metric}
\end{align}

\end{rmk}

\section{Blow down analysis}

For any $r>0$, define a new manifold $(M^{r},g^{r})$ through 
\[
\Phi_{r}:M\backslash K'\rightarrow M^{r}
\]
 is a diffeomorphism and 
\begin{equation}
g^{r}=\frac{1}{r^{2}}(\Phi_{r}^{-1})^{*}(g).\label{g^r}
\end{equation}
 If $\{x_{i}\}$ is the coordinate on $M\backslash K$, we set $ $
\begin{equation}
\bar{x}_{i}=\frac{1}{r}x_{i}(\Phi_{r}^{-1}).\label{x-bar}
\end{equation}
 Then actually $M^{r}$ is diffeomorphic to $\mathbb{R}^{3}\backslash B_{\frac{1}{r}}(0)$,
and the coordinates $\{\bar{x}_{i}\}_{i=1}^{3}$ can be regarded as
the Euclidean coordinates on $\mathbb{R}^{3}\backslash B_{\frac{1}{r}}(0)$.
When we take limit of the functions on $M^{r}$ or surfaces in $M^{r}$
(in Hausdorff sense or smooth sense) as $r\rightarrow\infty,$ actually
we identify $\{\bar{x}_{i}\}$ in different $M_{r}$. So the limit
function exists on $\mathbb{R}^{3}\backslash\{0\}$ which is the limit
space of $\mathbb{R}^{3}\backslash B_{\frac{1}{r}}(0)$ and the limit
surface (either in Hausdorff sense or in smooth sense) exists in $\mathbb{R}^{3}$
which can be regarded as the completion of $\mathbb{R}^{3}\backslash\{0\}$.
So we know 
\[
\lim_{r\rightarrow\infty}g_{\bar{i}\bar{j}}^{r}=\lim_{r\rightarrow\infty}\frac{1}{r^{2}}(\Phi_{r}^{-1})^{*}(g)(r(\Phi_{r})_{*}(\partial_{i}),r(\Phi_{r})_{*}(\partial_{i}))=\delta_{ij}=\delta_{\bar{i}\bar{j}}.
\]

Now we have the three blow-downs as in\cite{Qing-Tian-CMC}. First
we consider
\[
\Sigma^{\frac{2}{H}}=\Phi_{\frac{2}{H}}(\Sigma)\subset M^{\frac{2}{H}}.
\]

\begin{lem}\label{Blow down by H/2} Suppose that $\{\Sigma_{i}\}$
is a sequence of stable constant mean curvature spheres in a given
asymptotically flat end $(M\backslash K',g)$ and that

\begin{equation}
\lim_{i\rightarrow\infty}r_{0}(\Sigma_{i})=\infty.
\end{equation}
And suppose that $\Sigma_{i}$ separates the infinity from the compact
part $K'$. Then, there is a subsequence of $\{\Sigma_{i}^{\frac{2}{H}}\}$
which converges in Hausdorff sense to a round sphere $S_{1}^{2}(a)\subset(\mathbb{R}^{3},\delta)$
of radius $1$ and centered at $a\in\mathbb{R}^{3}$. Moreover, the
convergence is in $C^{2,\alpha}$ sense away from the origin.

Moreover if $\lim_{i\rightarrow\infty}\frac{r_{1}(\Sigma_{n})}{r_{0}(\Sigma_{n})}=+\infty$,
we have $|a|=1$, that is, the origin lies on the limit surface.

\end{lem}

\begin{proof} Suppose that there is a sequence of stable constant
mean curvature spheres $\{\Sigma_{i}\}$ such that

\[
\lim_{i\rightarrow\infty}r_{0}(\Sigma_{i})=\infty,
\]

we have known from Lemma \ref{Willmore functional} that 

\[
\lim_{i\rightarrow\infty}\int_{\Sigma_{i}}H_{e}^{2}d\mu_{e}=16\pi.
\]

Then from Theorem 3.1 of \cite{L.Simon-Willmore}, we can find a subsequence
which converge in Hausdorff sense to a genus $0$ surface, that is
a sphere. Because 
\[
\lim_{i\rightarrow\infty}|\Sigma_{i}^{\frac{2}{H}}|_{e}=\lim_{i\rightarrow\infty}|\Sigma_{i}^{\frac{2}{H}}|_{g^{\frac{2}{H}}}=\lim_{i\rightarrow\infty}\frac{1}{4}\int_{\Sigma_{i}}H^{2}d\mu=4\pi,
\]
the limit surface is a unit sphere. Away from the origin, the second
fundamental form and its derivative of $\Sigma_{i}^{\frac{2}{H}}$
have uniform bounds. So the convergence is $C^{2,\alpha}$ away from
the origin.

The second part follows from $|a|\leq1$ (because $\Sigma_{n}$ separates
$K'$ from infinity) and 
\[
\tilde{r}_{1}=\frac{H}{2}r_{1},\tilde{r}_{0}=\frac{H}{2}r_{0}.
\]

\end{proof}

Then, we use $r_{0}^{-1}$ to blow down the surface

\begin{equation}
\Sigma^{r_{0}}=\Phi_{r_{0}}(\Sigma)
\end{equation}

\begin{lem}\label{Blow down by r_0^{-1}} Suppose that $\{\Sigma_{i}\}$
is a sequence of stable constant mean curvature spheres in a given
asymptotically flat end $(\mathbb{R}^{3}\setminus B_{1}(0),g)$ and
that

\begin{equation}
\lim_{i\rightarrow\infty}r_{0}(\Sigma_{i})=\infty.
\end{equation}

And suppose that

\begin{equation}
\lim_{i\rightarrow\infty}r_{0}(\Sigma_{i})H(\Sigma_{i})=0.
\end{equation}

Then there is a subsequence of $\{\Sigma_{n}^{r_{0}}\}$ converges
to a 2-plane at distance $1$ from the origin. Moreover the convergence
is in $C^{2,\alpha}$ in any compact set of $\mathbb{R}^{3}$.

\end{lem}

\begin{proof} $\Sigma_{n}^{r_{0}}\subset M^{r_{0}(\Sigma_{n})}$
and in ${\rm dist_{\delta_{\bar{i}\bar{j}}}}(\Sigma_{n}^{r_{0}},0)=1.$
Note that from Theorem \ref{Main curvature estimates}, $\AA^{r_{0}}(\Sigma_{n}^{r_{0}})\rightarrow0.$
And $H^{r_{0}}(\Sigma_{n}^{r_{0}})=r_{0}H\rightarrow0.$ So we have
$A^{r_{0}}(\Sigma_{n}^{r_{0}})\rightarrow0.$ Here $\AA^{r_{0}},H^{r_{0}},A^{r_{0}}$
represent the trace free part of the second fundamental form and mean
curvature and second fundamental form with respect to metric $g^{r_{0}}$.
So we can find a subsequence of $\Sigma_{n}^{r_{0}}$ which converges
to a 2-plane at distance $1$ from the origin. From the same reason
as the last lemma, the convergence is $C^{2,\alpha}$ in any compact
set of $\mathbb{R}^{3}.$

\end{proof}

We must understand the behavior of the surfaces $\Sigma_{i}$ in the
scales between $r_{0}(\Sigma_{i})$ and $H^{-1}(\Sigma_{i})$. We
consider the scale $r_{i}$ such that

\begin{eqnarray}
\lim_{i\rightarrow\infty}\frac{r_{0}(\Sigma_{i})}{r_{i}}=0 &  & \lim_{i\rightarrow\infty}r_{i}H(\Sigma_{i})=0
\end{eqnarray}

and blow down the surfaces

\begin{equation}
\Sigma_{i}^{r_{i}}=\Phi_{r_{i}}(\Sigma_{i})
\end{equation}

\begin{lem}\label{Blow down by middle size} Suppose that $\{\Sigma_{i}\}$
is a sequence of stable constant mean curvature surfaces in a given
asymptotically flat end $(\mathbb{R}^{3}\setminus B_{1}(0),g)$ and
that

\begin{equation}
\lim_{i\rightarrow\infty}r_{0}(\Sigma_{i})=\infty
\end{equation}

And suppose that $r_{i}$ satisfies

\begin{eqnarray}
\lim_{i\rightarrow\infty}\frac{r_{0}(\Sigma_{i})}{r_{i}}=0 &  & \lim_{i\rightarrow\infty}r_{i}H(\Sigma_{i})=0\label{r_i property}
\end{eqnarray}

Then there is a subsequence of $\{\Sigma_{i}^{r_{i}}\}$ converges
to a 2-plane at the origin in Gromov-Hausdorff distance. Moreover
the convergence is $C^{2,\alpha}$ in any compact subset away from
the origin.

\end{lem}

\begin{proof} 
\begin{align*}
\int_{B_{R}}|A^{r_{i}}|^{2}d\mu_{g^{r_{i}}} & =\int_{B_{r_{i}R}}|A|^{2}d\mu\\
 & =\int_{B_{r_{i}R}}|\AA|^{2}d\mu+\frac{1}{2}\int_{B_{r_{i}R}}H^{2}d\mu\\
 & \leq C(r_{0}^{-q}+H^{2}R^{2}r_{i}^{2}).
\end{align*}
From (\ref{r_i property}) we know for any fixed $R>0$, 
\[
\int_{B_{R}}|A^{r_{i}}|^{2}d\mu_{g^{r_{i}}}\rightarrow0
\]
as $i\rightarrow\infty.$

From Lemma 2.1 in \cite{L.Simon-Willmore}, we can get the first part
of the conclusion. And again from the pointwise estimate of second
fundamental and its derivative, we get the $C^{2,\alpha}$ convergence
away from the origin. 

\end{proof}

\section{Asymptotically analysis}

In this section, we carry out Qing and Tian's harmonic map technique
in $C_{1,\tau}^{4}$-AF manifolds. In the end of this section we will
derive a strengthened estimate on the second fundamental form, Lemma
\ref{The new estimate for the second fundamental form}. First let
us revise the properties of harmonic function on a column. Suppose
$N=[0,3L]\times S^{1}$ for some constant $L$ to be fixed later.
Choose $(t,\theta)$ as the coordinates, where $0\leq t\leq3L,0\leq\theta\leq2\pi.$
Denote
\[
\|u\|_{1,i}^{2}=\int_{[(i-1)L,iL]\times S^{1}}(|u|^{2}+|\tilde{\nabla}u|^{2})dtd\theta,
\]
where $(t,\theta)$ is the standard column coordinate and $\tilde{\nabla}$
is the gradient operator with respect to the metric $dt^{2}+d\theta^{2}$.

\begin{lem}\label{three element lemma} Suppose $u\in W^{1,2}(N,R^{k})$
satisfies

\begin{equation}
\tilde{\Delta}u+A\cdot\tilde{\nabla}u+B\cdot u=h
\end{equation}
in $N$, where$N=[0,3L]\times S^{1}$ and $\tilde{\Delta}=\frac{\partial^{2}}{\partial t^{2}}+\frac{\partial^{2}}{\partial\theta^{2}},\tilde{\nabla}=(\frac{\partial}{\partial t},\frac{\partial}{\partial\theta})$
. And suppose that $L$ is given and large. Then there exists a positive
number $\delta_{0}$ such that if

\begin{equation}
\|h\|_{L^{2}(N)}\leq\delta_{0}\max_{1\leq i\leq3}\|u\|_{1,i}
\end{equation}
and

\begin{eqnarray}
\|A\|_{L^{\infty}(N)}\leq\delta_{0} &  & \|B\|_{L^{\infty}(N)}\leq\delta_{0}
\end{eqnarray}
then,

(a) $\|u\|_{1,3}\leq e^{-\frac{1}{2}L}\|u\|_{1,2}$ implies $\|u\|_{1,2}<e^{-\frac{1}{2}L}\|u\|_{1,1}$

(b) $\|u\|_{1,1}\leq e^{-\frac{1}{2}L}\|u\|_{1,2}$ implies $\|u\|_{1,2}<e^{-\frac{1}{2}L}\|u\|_{1,3}$

(c) If both $\int_{L\times S^{1}}ud\theta$ and $\int_{2L\times S^{1}}ud\theta\leq\delta_{0}\max_{1\leq i\leq3}\|u\|_{1,i}$,
then either $\|u\|_{1,2}<e^{-\frac{1}{2}L}\|u\|_{1,1}$ or $\|u\|_{1,2}<e^{-\frac{1}{2}L}\|u\|_{1,3}$\end{lem}

\begin{proof}

\end{proof}Suppose that $u\in W^{1,2}(\Sigma)$ and $u$ is harmonic,
we can deduce that if $u$ satisfies (a)(b)(c')with

(c') If both $\int_{L\times S^{1}}ud\theta$ and $\int_{2L\times S^{1}}ud\theta=0$,
then either $\|u\|_{1,2}<e^{-\frac{1}{2}L}\|u\|_{1,1}$ or $\|u\|_{1,2}<e^{-\frac{1}{2}L}\|u\|_{1,3}$

A harmonic function $u$ can be written as:

\[
u=a_{0}+b_{0}t+\sum_{n=1}^{\infty}\{e^{nt}(a_{n}\cos n\theta+b_{n}\sin n\theta)+e^{-nt}(a_{-n}\cos n\theta+b_{-n}\sin n\theta)\}
\]

Then it follows that:

\begin{eqnarray*}
\|u\|_{1,i}^{2} & = & 2\pi((a_{0}^{2}+b_{0}^{2})L+a_{0}b_{0}L^{2}(2i-1)+\frac{1}{3}b_{0}^{2}L^{3}(3i^{2}-3i+1))\\
 &  & +\frac{\pi}{2}\sum_{n=1}^{\infty}\{\frac{e^{2nL-1}}{n}(e^{2(i-1)nL}(a_{n}^{2}+b_{n}^{2})+e^{-2niL}(a_{-n}^{2}+b_{-n}^{2}))+4L(a_{n}a_{-n}+b_{n}b_{-n})\}\\
 &  & +\pi\sum_{n=1}^{\infty}\{\frac{e^{2nL-1}}{n}(e^{2(i-1)nL}(n^{2}a_{n}^{2}+n^{2}b_{n}^{2})+e^{-2niL}(n^{2}a_{-n}^{2}+n^{2}b_{-n}^{2}))\\
 &  & +4L(n^{2}a_{n}a_{-n}+n^{2}b_{n}b_{-n})\}
\end{eqnarray*}
 $i=1,2,3$

If $L$ is fixed and sufficiently large, then we have
\[
\|u\|_{1,2}^{2}<\frac{1}{2}(e^{L}\|u\|_{1,3}^{2}+e^{-L}\|u\|_{1,1}^{2})
\]

which implies (a). We get (b) in the same way. For (c'), we have $a_{0}=b_{0}=0$
then we have

\[
\|u\|_{1,2}^{2}<\frac{1}{2}e^{-L}(\|u\|_{1,3}^{2}+\|u\|_{1,1}^{2})
\]
which implies either $\|u\|_{1,2}<e^{-\frac{1}{2}L}\|u\|_{1,1}$ or
$\|u\|_{1,2}<e^{-\frac{1}{2}L}\|u\|_{1,3}$.

The second step is to pass limits. If the proposition were false,
then one would find a sequence of $\delta_{k}\rightarrow0$ and a
sequence of solution $u_{k}$ with $\|h_{k}\|_{L^{2}}\leq\delta_{k}\max_{1\leq i\leq3}\|u_{k}\|_{1,i}$,
$\|A_{k}\|_{\infty}\leq\delta_{k}$ and $\|B_{k}\|_{\infty}\leq\delta_{k}$
solves:

\begin{eqnarray*}
\tilde{\Delta}u_{k}+A_{k}\cdot\tilde{\nabla}u_{k}+B_{k}\cdot u_{k} & = & h_{k}.
\end{eqnarray*}

And $u_{k}$ violate (a)(b) or (c). We may assume $\max_{1\leq i\leq3}\|u_{k}\|_{1,i}=1$
otherwise we can normalize $u_{k}$. So we know $\|u_{k}\|_{1,2}>C>0$
for a uniform $C$ because $u_{k}$ violate (a)(b) or (c). Then we
know there is a subsequence that converges to some harmonic function
$u\in W^{1,2}(\Sigma)$ weakly. From the interior $W^{2,p}$ estimate
we know the convergence is strongly $W^{1,2}$ in $I_{2}$, which
implies that $u$ is not trivially zero.

And because $u_{i}\rightharpoonup u$ weakly in $W^{1,2}(\Sigma)$
sense. So $u_{i}\rightharpoonup u$ in $W^{1,2}(I_{1})$ and $W^{1,2}(I_{3})$
sense, then we have:

\begin{align*}
\liminf_{i\rightarrow\infty}\|u_{i}\|_{1,1} & \geq\|u\|_{1,1}\\
\liminf_{i\rightarrow\infty}\|u_{i}\|_{1,3} & \geq\|u\|_{1,3}
\end{align*}
\[
\]
and

\[
\lim_{i\rightarrow\infty}\|u_{i}\|_{1,2}=\|u\|_{1,2}
\]

then $u_{i}$ converges to some non-trivial harmonic function $u$
which violates one of (a)(b) or (c'). So we proved the lemma. 

Given a surface $\Sigma$ in $R^{3}$. Recall 
\[
\Delta_{e}v_{e}+|\nabla_{e}v_{e}|^{2}v_{e}=\nabla_{e}H_{e},
\]
where $v_{e}$ is the Gauss map $\Sigma\rightarrow S^{2}$. For the
constant mean curvature spheres in the asymptotically flat end $(R^{3}\setminus B_{1}(0),g)$,
we have

\begin{lem}
\[
|\nabla_{e}H_{e}|(x)\leq C|x|^{-3}
\]
\end{lem}

\begin{proof}Because of the uniform equivalence of the metric $g$
and the euclidean metric, we can prove:
\[
|\nabla H_{e}|(x)\leq C|x|^{-3}
\]
instead. From the expression of $H-H_{e}$ (\ref{H-H_e}), we have 

\begin{eqnarray}
 &  & |\nabla H_{e}|\leq|\overline{\nabla}h_{ij}||A|+|h_{ij}||A|^{2}+|h_{ij}||\nabla\AA_{ij}|+H|A||h_{ij}|+H|\overline{\nabla}h_{ij}|\nonumber \\
 &  & +|A||\overline{\nabla}h_{ij}|+|\overline{\nabla}^{2}h|\nonumber \\
 &  & \leq C|x|^{-3}
\end{eqnarray}

\end{proof}

Suppose $\Sigma$ is a stable CMC sphere in the asymptotically flat
end which separates $K'$ from infinity. We are interested in the
case when $r_{0}$ is much smaller than $H^{-1}.$

Set

\[
A_{r_{1},r_{2}}=\{x\in\Sigma:r_{1}\leq|x|\leq r_{2}\}
\]
and $A_{r_{1},r_{2}}^{0}$ stands for the standard annulus in $\mathbb{R}^{2}$.
Consider the behavior of the normal vector $v$ on $A_{Kr_{0}(\Sigma),sH^{-1}(\Sigma)}$
of $\Sigma$ where $K$ will be fixed large and $s$ will be fixed
small. The lemma below gives us a good coordinate on the surface.

\begin{lem} Suppose $\Sigma$ is a stable constant mean curvature
sphere in a given asymptotically flat end $(M\backslash K',g)$ which
separates $K'$ from infinity. Then, for any $\varepsilon>0$ and
$L$ fixed, there are $M$,$s$ and $K$ such that, if $r_{0}\geq M$
and $Kr_{0}(\Sigma)<r<sH^{-1}(\Sigma)$, then $(r^{-1}A_{r,e^{L}r},r^{-2}g_{e})$
may be represented as $(A_{1,e^{L}}^{0},\overline{g})$ and

\begin{equation}
\|\overline{g}-|dx|^{2}\|_{C^{1,\alpha}(A_{1,e^{L}}^{0})}\leq\varepsilon.
\end{equation}

In other words, in the cylindrical coordinates $[\log r,L+\log r]\times S^{1}$

\begin{equation}
\|\overline{g}_{c}-(dt^{2}+d\theta^{2})\|_{C^{1,\alpha}(S^{1}\times[\log r,L+\log r])}\leq\varepsilon
\end{equation}
\end{lem}

\begin{proof}By contradiction, we assume this lemma were false. So
from Lemma \ref{Blow down by r_0^{-1}}, for some $\varepsilon_{0}>0$,
there exists a sequence $\Sigma_{n}$ with $r_{0}(\Sigma_{n})\rightarrow\infty$
and $\tilde{l}_{n}\rightarrow\infty$ such that 
\[
((Kr_{0}e^{\tilde{l}_{n}L})^{-1}A_{Kr_{0}e^{\tilde{l}_{n}L},Kr_{0}e^{(\tilde{l}_{n}+1)L}},(Kr_{0}e^{\tilde{l}_{n}L})^{-2}g_{e})
\]
is not within $\varepsilon_{0}$ neighborhood of $(A_{1,e^{L}}^{0},\bar{g})$
in $C^{1,\alpha}$ sense.

From Lemma \ref{Blow down by H/2}, we know: if we fix a small $s$, 

\[
\frac{Kr_{0}e^{\tilde{l}_{n}L}}{sH^{-1}(\Sigma_{n})}\rightarrow0.
\]

So if we let $r_{n}=Kr_{0}e^{\tilde{l}_{n}L}$, then 
\[
\lim_{n\rightarrow\infty}\frac{r_{n}}{Kr_{0}}=\infty,\lim_{n\rightarrow\infty}\frac{r_{n}}{sH^{-1}}=0.
\]
However if we blowdown the surface by $r_{n}^{-1},$ we get a contradiction
with Lemma \ref{Blow down by middle size}. So we have proved the
lemma.

\end{proof}

Now consider the cylindrical coordinates $(t,\theta)$ on $(S^{1}\times[\log Kr_{0},\log sH^{-1}])$,
then the tension field

\begin{equation}
|\tau(v)|=r^{2}|\nabla_{e}H_{e}|\leq Cr^{-1}
\end{equation}
for $t\in[\log Kr_{0},\log sH^{-1}]$. Thus, 
\begin{equation}
\int_{S^{1}\times[t,t+L]}|\tau(v)|^{2}dtd\theta\leq Cr^{-2}
\end{equation}

Let $I_{i}$ stand for $S^{1}\times[\log Kr_{0}+(i-1)L,\log Kr_{0}+iL]$,
and $N_{i}$ stand for $I_{i-1}\cup I_{i}\cup I_{i+1}$. On $\Sigma_{n}$
we assume $\log(sH^{-1})-\log(Kr_{0})=l_{n}L$. 

Now we prove the energy decay. Suppose $f_{ij}$ is the metric of
the surface $\Sigma_{n}$ , i.e. the restriction of $g_{ij}$ to $\Sigma_{n}$$.$
For sufficiently large $K$, we consider $(\Sigma_{n}\cap B_{Kr_{0}}^{c}(0),f_{ij}|x|^{-4}(Kr_{0})^{2})$
which is close to the unit ball of $\mathbb{R}^{2}$. The Gauss map
$v_{n}:\Sigma_{n}\rightarrow S^{2}$ induces a map $\hat{v}_{n}:B_{1}(0)\rightarrow S^{2}$.
Note that the energy of $\hat{v}$ will concentrate at the origin
of $B_{1}(0)$ and the tension field $\hat{\tau}$ of the map $\hat{v}$
satisfies
\[
|\hat{\tau}|\leq C|x|^{-3}|x|^{4}(Kr_{0})^{-2}=C|x|(Kr_{0})^{-2}=\frac{C(Kr_{0})^{-1}}{\sqrt{4s^{2}e^{-2l_{n}L}+\hat{r}^{2}}}
\]
 where $\hat{r}$ denotes the radius function of the unit ball. First
we have:

\begin{lem}For every $i\in[3,l_{n}-2]$, there exists a geodesic
$\gamma_{i}$ such that:

\[
\int_{I_{i}}|\tilde{\nabla}(v_{n}-\gamma_{i})|^{2}dtd\theta\leq C(e^{-iL}+e^{-(l_{n}-i)L})(s^{2}+r_{0}^{-1}).
\]

\end{lem}

\begin{proof}First we have 
\begin{align*}
[v_{n}]_{C^{\alpha}(I_{i})} & \leq C\|\tilde{\nabla}v\|_{L^{\infty}}\\
 & \leq CKr_{0}e^{iL}(\|\AA\|_{L^{\infty}}+H)\\
 & \leq C(r_{0}^{-\frac{1}{2}}+s).
\end{align*}
from the estimate of the second fundamental form. So when $r_{0}$
is sufficient large and $s$ is sufficiently small, $[v_{n}]_{C^{\alpha}(I_{i})}$
is sufficiently small.

Note that $S^{2}$ is smooth compact manifold and ${\rm osc_{I_{i}}}v_{n}$
is very small. So for each $I_{i}$ we can choose two points $P_{i}$
and $Q_{i}$ on $S^{2}$ such that
\begin{align*}
|P_{i}-\frac{1}{2\pi}\int_{(i-1)L\times S^{1}}v_{n}d\theta| & \leq C\max_{(i-1)L\times S^{1}}|v_{n}-P_{i}|^{2}\\
|Q_{i}-\frac{1}{2\pi}\int_{iL\times S^{1}}v_{n}d\theta| & \leq C\max_{iL\times S^{1}}|v_{n}-Q_{i}|^{2}.
\end{align*}
Also we know easily $dist_{S^{2}}(P_{i},Q_{i})\leq C(r_{0}^{-\frac{1}{2}}+s).$
So we can choose one unique geodesic $\gamma_{i}$ which joins $P_{i}$
and $ $$Q_{i}$. 

If we regard $\gamma_{i}$ as a harmonic map from $ $$[(i-1)L,iL]\times S^{1}$
to $S^{2}$, we can extend $\gamma_{i}$ to $[\log(Kr_{0}),\log(sH^{-1})]\times S^{1}$.
$u_{n,i}=v_{n}-\gamma_{i}$ satisfies 
\[
\tilde{\Delta}u_{n,i}+A_{n,i}\cdot\tilde{\nabla}u_{n,i}+B_{n,i}\cdot u_{n,i}=\tau_{n}
\]
where 
\begin{align}
|A_{n,i}| & \leq C(|\tilde{\nabla}v_{n}|+|\tilde{\nabla}\gamma_{i}|)\leq\delta_{0},\nonumber \\
|B_{n,i}| & \leq C\min\{|\tilde{\nabla}v_{n}|^{2},|\tilde{\nabla}\gamma_{i}|^{2}\}\leq\delta_{0}.\label{eq:|A| and |B| inequality}
\end{align}

To use Lemma \ref{three element lemma} (C), we have to verify that
\[
\|\tau_{n}\|_{L^{2}(N_{i})}\leq\delta_{0}\max_{i-1\leq k\leq i+1}\|u_{n,i}\|_{1,k}
\]
 and 
\begin{align*}
\int_{(i-1)L\times S^{1}}u_{n,i}d\theta & \leq\delta_{0}\max_{i-1\leq k\leq i+1}\|u_{n,i}\|_{1,k}\\
\int_{iL\times S^{1}}u_{n,i}d\theta & \leq\delta_{0}\max_{i-1\leq k\leq i+1}\|u_{n,i}\|_{1,k}
\end{align*}
where $N_{i}=I_{i-1}\cup I_{i}\cup I_{i+1}.$ 

However we have 
\[
\int_{(i-1)L\times S^{1}}u_{n,i}d\theta\leq|2\pi P_{i}-\int_{(i-1)L\times S^{1}}v_{n}d\theta|\leq C\max_{(i-1)L\times S^{1}}|v_{n}-P_{i}|^{2}.
\]
 By interior estimate and trace embedding we have
\[
\max_{(i-1)\times S^{1}}|v_{n}-P_{i}|\leq C(\max_{i-1\leq k\leq i+1}\|u_{n,i}\|_{1,k}+\|\tau_{n}\|_{L^{2}(N_{i})}).
\]
So if we have 
\[
\|\tau_{n}\|_{L^{2}(N_{i})}\leq\delta_{0}\max_{i-1\leq k\leq i+1}\|u_{n,i}\|_{1,k}
\]
we will have 
\begin{align*}
\int_{(i-1)L\times S^{1}}u_{n,i}d\theta & \leq C\max_{(i-1)L\times S^{1}}|v_{n}-P_{i}|\max_{i-1\leq k\leq i+1}\|u_{n,i}\|_{1,k}\\
 & \leq C\delta_{0}\max_{i-1\leq k\leq i+1}\|u_{n,i}\|_{1,k}.
\end{align*}
and in the same way we will get
\[
\int_{iL\times S^{1}}u_{n,i}d\theta\leq\delta_{0}\max_{i-1\leq k\leq i+1}\|u_{n,i}\|_{1,k}.
\]
So if we cannot use Lemma \ref{three element lemma} (C), the only
reason is 
\[
\max_{i-1\leq k\leq i+1}\|u_{n,i}\|_{1,k}\leq\delta_{0}\|\tau_{n}\|_{L^{2}(N_{i})},
\]
which will imply
\begin{align*}
\int_{I_{i}}|\tilde{\nabla}(v_{n}-\gamma_{i})|^{2}dtd\theta & \leq Ce^{-2t}\leq Ce^{-iL}r_{0}^{-1}\\
 & \leq C(e^{-iL}+e^{-(l_{n}-i)L})(s^{2}+r_{0}^{-1}).
\end{align*}
If (C) can be used, so we use it on $N_{i}$ for $u_{n,i}$. We have
\[
\|u_{n,i}\|_{1,i}<e^{-\frac{1}{2}L}\|u\|_{1,i-1}
\]
or 
\[
\|u_{n,i}\|_{1,i}<e^{-\frac{1}{2}L}\|u\|_{1,i+1}.
\]
Without loss of generality, we assume the first one happens. Then
we can push this relationship left and continue to use Lemma \ref{three element lemma}
(a) because (\ref{eq:|A| and |B| inequality}) always holds. If for
some $j\geq2,$ the theorem can be used until $N_{j+1},$ but not
until $N_{j}$, then we have 
\begin{align*}
\|u_{n,i}\|_{1,i} & <e^{-\frac{1}{2}(i-j)L}\|u\|_{1,j}\leq Ce^{-\frac{1}{2}(i-j)L}e^{-\frac{1}{2}jL}r_{0}^{-\frac{1}{2}}\\
 & \leq Ce^{-\frac{1}{2}iL}r_{0}^{-\frac{1}{2}}.
\end{align*}
If we can use (a) until $N_{2}$, then we have:

\begin{align*}
e^{\frac{L}{2}}\|u_{n,i}\|_{1,2} & \leq\|u_{n,i}\|_{1,1}=(\int_{I_{1}}u_{n,i}^{2}dtd\theta)^{\frac{1}{2}}+(\int_{I_{1}}|\tilde{\nabla}u_{n,i}|^{2}dtd\theta)^{\frac{1}{2}}\\
 & \leq(\int_{I_{2}}u_{n,i}^{2}dtd\theta)^{\frac{1}{2}}+(\int_{I_{1}}(u(t,\theta)-u(t+L,\theta))^{2}dtd\theta)^{\frac{1}{2}}+(\int_{I_{1}}|\tilde{\nabla}u_{n,i}|^{2}dtd\theta)^{\frac{1}{2}}.
\end{align*}
So we have
\begin{align*}
(e^{\frac{L}{2}}-1)\|u_{n,i}\|_{1,2} & \leq(\int_{I_{1}}(\int_{0}^{L}|\frac{\partial u_{n,i}}{\partial t}(t+s,\theta)|ds)^{2}dtd\theta)^{\frac{1}{2}}+(\int_{I_{1}}|\tilde{\nabla}u_{n,i}|^{2}dtd\theta)^{\frac{1}{2}}\\
 & \leq\int_{0}^{L}\int_{I_{1}}|\frac{\partial u_{n,i}}{\partial t}(t+s,\theta)|{}^{2}dtd\theta)^{\frac{1}{2}}ds+(\int_{I_{1}}|\tilde{\nabla}u_{n,i}|^{2}dtd\theta)^{\frac{1}{2}}\\
 & \leq C(\int_{I_{1}\cup I_{2}}|\tilde{\nabla}u_{n,i}|^{2}dtd\theta)^{\frac{1}{2}}\\
 & \leq C((\int_{I_{1}\cup I_{2}}|\tilde{\nabla}v_{n}|^{2}dtd\theta)^{\frac{1}{2}}+(\int_{I_{1}\cup I_{2}}|\tilde{\nabla}\gamma_{i}|^{2}dtd\theta)^{\frac{1}{2}})\\
 & \leq C(r_{0}^{-\frac{1}{2}}+s).
\end{align*}
So we have 
\[
\|u_{n,i}\|_{1,i}\leq Ce^{-\frac{i-2}{2}L}\|u_{n,i}\|_{1,2}\leq Ce^{-\frac{i}{2}L}(r_{0}^{-\frac{1}{2}}+s).
\]
If $ $$\|u_{n,i}\|_{1,i}<e^{-\frac{1}{2}L}\|u\|_{1,i+1}$ happens,
similarly, we have 
\[
\|u_{n,i}\|_{1,i}\leq Ce^{-\frac{l_{n}-i}{2}L}(r_{0}^{-\frac{1}{2}}+s).
\]
At last we get 
\[
\|u_{n,i}\|_{1,i}\leq C(e^{-\frac{i}{2}L}+e^{-\frac{l_{n}-i}{2}L})(r_{0}^{-\frac{1}{2}}+s).
\]

\end{proof}

From this lemma we have 
\[
\int_{[(i-1)L,iL]\times S^{1}}|\partial_{\theta}v_{n}|^{2}dtd\theta\leq C(e^{-iL}+e^{-(l_{n}-i)L})(r_{0}^{-1}+s^{2})
\]

To get the energy decay we use the Pohozaev equality. See for example
\cite{Lin-Wang-Harmonic-map-flow}, Lemma 2.4.

\begin{lem}Let $v$ be a solution to 
\[
\Delta_{e}v+|\nabla_{e}v|^{2}v=\hat{\tau}.
\]
And $v$ is defined on the disk $B_{r_{0}}$. Then we have 
\[
\int_{\partial B_{r_{0}}}(|\partial_{r}v|^{2}-r^{-2}|\partial_{\theta}v|^{2})d\mu(\partial B_{r_{0}})=\frac{2}{r_{0}}\int_{B_{r_{0}}}\hat{\tau}\cdot(x\nabla v)dx.
\]

\end{lem}

This lemma imply for $t\in[(i-1)L,iL]$ 
\begin{align*}
\int_{[(i-1)L,iL]\times S^{1}}|\partial_{t}v_{n}|^{2}dtd\theta & \leq\int_{[(i-1)L,iL]\times S^{1}}|\partial_{\theta}v_{n}|^{2}dtd\theta+C\int_{B_{e^{-iL}}}|\hat{\tau}|(\hat{x}\hat{\nabla}v_{n})d\hat{x}\\
 & \leq\int_{[(i-1)L,iL]\times S^{1}}|\partial_{\theta}v_{n}|^{2}dtd\theta\\
 & +C(\int_{B_{e^{-iL}}}|\hat{\tau}|^{2}|\hat{x}|^{2}d\hat{x})^{\frac{1}{2}}(\int_{B_{e^{-iL}}}|\hat{\nabla}v_{n}|^{2}d\hat{x})^{\frac{1}{2}}\\
 & \leq\int_{[(i-1)L,iL]\times S^{1}}|\partial_{\theta}v_{n}|^{2}dtd\theta+Ce^{-iL}(r_{0}^{-\frac{1}{2}}+s)\\
 & \leq C(e^{-iL}+e^{-(l_{n}-i)L})(r_{0}^{-\frac{1}{2}}+s).
\end{align*}
At last we get energy decay

\begin{lem}\label{estimate on the normal vector in the neck} Suppose
that $\{\Sigma_{n}\}$ is a sequence of stable constant mean curvature
spheres in a given asymptotically flat end $(\mathbb{R}^{3}\setminus B_{1}(0),g)$
which separate $K'$ from infinity and that

\begin{equation}
\lim_{i\rightarrow\infty}r_{0}(\Sigma_{n})=\infty
\end{equation}

And suppose that

\begin{equation}
\lim_{n\rightarrow\infty}r_{0}(\Sigma_{n})H(\Sigma_{n})=0
\end{equation}

Then for any $K>0$, $s>0$, there exists a uniform $C>0$ and $n_{0}$
such that, when $n\geq n_{0}$,

\begin{equation}
\max_{I_{i}}|\tilde{\nabla}v_{e}|\leq C(e^{-\frac{i}{2}L}+e^{-\frac{(l_{n}-i)}{2}L})(s+r_{0}^{-\frac{1}{2}})
\end{equation}
where

\begin{equation}
I_{i}=S^{1}\times[\log(Kr_{0}(\Sigma_{n}))+(i-1)L,\log(Kr_{0}(\Sigma_{n}))+iL]
\end{equation}
and

\begin{eqnarray}
i\in[0,l_{n}] &  & \log(Kr_{0}(\Sigma_{n}))+l_{n}L=\log(sH^{-1}(\Sigma_{n}))
\end{eqnarray}

\end{lem}

\begin{lem}\label{no-neck lemma} For any $\varepsilon>0$, there
is some $\delta>0$ and $M>0$ such that if $0<s<\delta$ and $n>M$
we have 
\[
{\rm OSC}{}_{\Sigma_{n}\cap B_{sH^{-1}}(0)}v_{n}\leq\varepsilon.
\]
\end{lem}

\begin{proof}Suppose $\log sH^{-1}-\log Kr_{0}=l_{n}L,$
\begin{align*}
{\rm OSC}{}_{\Sigma_{n}\cap B_{sH^{-1}}(0)}v_{n} & \leq\sum_{i=1}^{l_{n}}{\rm OSC}{}_{\Sigma_{n}\cap(B_{Kr_{0}e^{iL}}\backslash B_{Kr_{0}e^{(i-1)L}})}v_{n}+{\rm OSC}{}_{\Sigma_{n}\cap B_{Kr_{0}}}v_{n}\\
 & \leq C\sum_{i=1}^{l_{n}}(e^{-\frac{iL}{2}}+e^{-\frac{(l_{n}-i)L}{2}})(r_{0}^{-\frac{1}{2}}+s)^{\frac{1}{2}}+{\rm OSC}{}_{\Sigma_{n}\cap B_{Kr_{0}}}v_{n}\\
 & \leq C(r_{0}^{-\frac{1}{2}}+s)^{\frac{1}{2}}+{\rm OSC}{}_{\Sigma_{n}\cap B_{Kr_{0}}}v_{n}
\end{align*}
We choose $\delta$ small and $n$ large, so we have $r_{0}^{-\frac{1}{2}}+s$
is sufficiently small and from Lemma \ref{Blow down by r_0^{-1}},
for fixed $K$ and sufficiently large $n$, ${\rm OSC}{}_{\Sigma_{n}\cap B_{Kr_{0}}}v_{n}$
is also small. So we have proved this lemma

\end{proof}

From the lemma above we know the two limit planes we got in Lemma
\ref{Blow down by r_0^{-1}} and Lemma \ref{Blow down by middle size}
have the same normal vector.

\begin{lem}\label{The new estimate for the second fundamental form}
If $\Sigma$ is a stable CMC sphere in the asymptotically flat end,
then the second fundamental form of $\Sigma$ has the following estimate:
For a point $x\in(B_{Kr_{0}e^{(i+1)L}}\setminus B_{Kr_{0}e^{iL}})\cap\Sigma$,
\[
|A_{e}(x)|\leq C|x|^{-1}(e^{-\frac{i}{2}L}+e^{-\frac{(l_{n}-i)}{2}L})(s+r_{0}^{-\frac{1}{2}})^{\frac{1}{2}}
\]
where $sH^{-1}=Kr_{0}\cdot e^{l_{n}L}$.

\end{lem}

\begin{proof}Note that
\[
|A_{e}(x)|\leq C|\nabla_{e}v_{e}(x)|\leq C|x|^{-1}\sup_{I_{i}}|\hat{\nabla}v_{e}|\leq C|x|^{-1}(e^{-\frac{i}{2}L}+e^{-\frac{(l_{n}-i)}{2}L})(s+r_{0}^{-\frac{1}{2}})^{\frac{1}{2}}
\]

\end{proof}

\begin{cor}\label{Choose a mean value of the normal vector} Assume
the same condition as Proposition\ref{estimate on the normal vector in the neck}.
Choose some $p_{n}\in I_{\frac{l_{n}}{2}}$. Then

\begin{equation}
\sup_{x\in I_{i}}|v_{n}(x)-v_{n}(p_{n})|\leq C(e^{-\frac{1}{2}iL}+e^{-\frac{1}{4}l_{n}L})(s+r_{0}^{-\frac{1}{2}})^{\frac{1}{2}}
\end{equation}
for $i\in[0,\frac{1}{2}l_{n}]$

\begin{equation}
\sup_{x\in I_{i}}|v_{n}(x)-v_{n}(p_{n})|\leq C(e^{-\frac{1}{4}l_{n}L}+e^{-\frac{1}{2}(l_{n}-i)L})(s+r_{0}^{-\frac{1}{2}})^{\frac{1}{2}}
\end{equation}
for $i\in[\frac{1}{2}l_{n},l_{n}]$\end{cor}

\begin{proof}We only prove the first one. 
\begin{align*}
\sup_{I_{i}}|v_{n}(x)-v_{n}(p_{n})| & \leq\sum_{k=i}^{l_{n}/2}{\rm OSC}{}_{I_{k}}v_{n}\\
 & \leq C\sum_{k=i}^{l_{n}/2}(e^{-\frac{k}{2}L}+e^{-\frac{(l_{n}-k)}{2}L})(s+r_{0}^{-\frac{1}{2}})^{\frac{1}{2}}\\
 & \leq\frac{C}{1-e^{-\frac{L}{2}}}(e^{-\frac{i}{2}L}+e^{-\frac{l_{n}}{4}L})(s+r_{0}^{-\frac{1}{2}})^{\frac{1}{2}}.
\end{align*}
The second one follows similarly. 

\end{proof}

\section{Mass integral}

In this section we prove Theorem \ref{thm1}. To detect the non 0
mass we use the integral below, for some constant vector $b$ to be
fixed,

\[
\int_{\Sigma}(H-H_{e})<v_{e},b>_{e}d\mu_{e}.
\]
First we have
\begin{align*}
\int_{\Sigma}H & <v_{e},b>_{e}d\mu_{e}=H\int_{int(\Sigma)}{\rm div}(b)d\mu_{e}=0.
\end{align*}
$\int_{\Sigma}H_{e}<v_{e},b>_{e}d\mu_{e}$ is the variation of the
area of $\Sigma$ in the direction of $b,$ which is also $0.$ So
we have 
\begin{equation}
\int_{\Sigma}(H-H_{e})<v_{e},b>_{e}d\mu_{e}=0.\label{trivial identity}
\end{equation}

Should Theorem \ref{thm1} be false, we could find a sequence of stable
CMC spheres $\{\Sigma_{n}\}$ which separate $K'$ from infinity,
with
\[
\lim_{n\rightarrow\infty}r_{0}(\Sigma_{n})=+\infty
\]
\[
\lim_{n\rightarrow\infty}\int_{\Sigma_{n}}(H-H_{e})<v_{e},b>_{e}d\mu_{e}\neq0
\]
which is a contradiction with (\ref{trivial identity}). 

To do this, first suppose Theorem \ref{thm1} were false, then we
could find a subsequence of stable CMC spheres $\{\Sigma_{n}\}$ which
separate $K'$ from infinity such that 
\[
r_{1}(\Sigma_{n})/r_{0}(\Sigma_{n})>n
\]
 and 
\[
\lim_{n\rightarrow\infty}r_{0}(\Sigma_{n})=+\infty.
\]
From Lemma \ref{The relationship of H and r1} we know 
\[
\lim_{n\rightarrow\infty}r_{0}(\Sigma_{n})H(\Sigma_{n})=0.
\]

From Lemma \ref{The relationship of H and r1} and Lemma \ref{Blow down by H/2},
$\Sigma_{n}^{\frac{2}{H}}$ converge to some sphere $S_{1}(a)$, with
center to be unit vector $a$. Then the origin lies on $S_{1}(a).$ 

Choose 
\[
b=-a,
\]
 and we consider the integral: 
\begin{eqnarray}
 &  & \int_{\Sigma}(H-H_{e})<v_{e},b>_{e}d\mu_{e}=\int_{\Sigma}(-f^{ik}h_{kl}f^{lj}A_{ij}+\frac{1}{2}Hv^{i}v^{j}h_{ij}-f^{ij}v^{l}\overline{\nabla}_{i}h_{jl}\nonumber \\
 &  & +\frac{1}{2}f^{ij}v^{l}\overline{\nabla}_{l}h_{ij}\pm C|h||\overline{\nabla}h|\pm C|h|^{2}|A|)<v_{e},b>_{e}d\mu_{e}+O(r_{0}^{-1}),
\end{eqnarray}
here $i,j$ ran from $1$ to 3, and $f_{ij}$ is the restriction of
$g_{ij}$ on $\Sigma$. 

\begin{rmk} Note that $A_{ij}-(A_{e})_{ij}=O(|x|^{-2})$ and $v_{e}-v=O(|x|^{-1})$.
From 
\[
\int_{\Sigma}|x|^{-3}d\mu_{e}=O(r_{0}^{-1}).
\]
 We may identify $A_{ij}$ with $(A_{e})_{ij}$ and $v$ with $v_{e}$
in the integral where this is needed.

\end{rmk}

From
\begin{eqnarray}
 &  & \int_{\Sigma_{n}}-f^{ij}v^{l}(\overline{\nabla}_{i}h_{jl})v^{m}b^{m}d\mu_{e}\nonumber \\
 & = & \frac{1}{2}\int_{\Sigma_{n}}(f^{ij}h_{jk}f^{kl}A_{li}-Hv^{j}v^{l}h_{jl})v^{m}b^{m}d\mu_{e}+\frac{1}{2}\int_{\Sigma_{n}}f^{ij}v^{l}h_{jl}A_{ik}f^{km}b^{m}d\mu_{e}\nonumber \\
 &  & -\frac{1}{2}\int_{\Sigma_{n}}f^{ij}v^{l}(\overline{\nabla}_{i}h_{jl})v^{m}b^{m}d\mu_{e},
\end{eqnarray}
we change the integral into:

\begin{eqnarray}
\int_{\Sigma_{n}}(H-H_{e})<v_{e},b>_{e}d\mu_{e}=\int_{\Sigma_{n}}(-\frac{1}{2}f^{ik}h_{kl}f^{lj}A_{ij}v^{m}b^{m}+\frac{1}{2}f^{ij}v^{l}h_{jl}A_{ik}f^{km}b^{m}\nonumber \\
-\frac{1}{2}f^{ij}v^{l}\overline{\nabla}_{i}h_{jl}v^{m}b^{m}+\frac{1}{2}f^{ij}v^{l}\overline{\nabla}_{l}h_{ij}v^{m}b^{m})d\mu_{e}+O(r_{0}^{-1})\nonumber \\
\end{eqnarray}

At $x\in\Sigma_{n}$ we choose a frame $\{e_{1},e_{2},v\},$ where
$e_{\alpha},\alpha=1,2$ (the Greek indices runs from $1,2$) form
the orthonormal basis of $T_{x}\Sigma_{n}$. We have $f_{\alpha\beta}=f^{\alpha\beta}=\delta_{\alpha\beta}$
and for any tensor $p_{ij}$, we have 
\[
p_{\alpha\alpha}+p(v,v)=g^{ij}p_{ij}=p_{ii}+O(r^{-1})|p|.
\]
 Sometimes we denote $v$ direction by $3',$ so $f_{\alpha,3'}=f^{\alpha,3'}=f_{3'3'}=f^{3'3'}=0$
and $A_{3'3'}=A_{\alpha,3'}=0$ and $v^{3'}=1,v^{\alpha}=0$.

So we know 
\begin{align*}
-\frac{1}{2}f^{ik}h_{kl}f^{lj}A_{ij}v^{m}b^{m} & =-\frac{1}{2}f^{\alpha\gamma}f^{\beta\eta}h_{\gamma\eta}A_{\alpha\beta}v^{m}b^{m}-\frac{1}{2}f^{3'3'}f^{3'3'}h_{3'3'}A_{3'3'}\\
 & =-\frac{1}{2}f^{\alpha\gamma}f^{\beta\eta}h_{\gamma\eta}A_{\alpha\beta}v^{m}b^{m}\\
 & =-\frac{1}{2}h_{\alpha\beta}A_{\alpha\beta}v^{m}b^{m}=-\frac{1}{2}h_{\alpha\beta}A_{\alpha\beta}b^{3'}\\
\frac{1}{2}f^{ij}v^{l}h_{jl}A_{ik}f^{km}b^{m} & =\frac{1}{2}f^{\beta\gamma}v^{3'}h_{\gamma3'}A_{\beta\eta}f^{\eta\alpha}b^{\alpha}=\frac{1}{2}h_{\beta3'}A_{\alpha\beta}b^{\alpha}\\
-\frac{1}{2}f^{ij}v^{l}\overline{\nabla}_{i}h_{jl}v^{m}b^{m} & =-\frac{1}{2}f^{\alpha\beta}v^{3'}\overline{\nabla}_{\alpha}h_{\beta3'}v^{3'}b^{3'}=-\frac{1}{2}\overline{\nabla}_{\alpha}h_{\alpha3'}b^{3'}\\
\frac{1}{2}f^{ij}v^{l}\overline{\nabla}_{l}h_{ij}v^{m}b^{m} & =\frac{1}{2}f^{\alpha\beta}v^{3'}\overline{\nabla}_{3'}h_{\alpha\beta}v^{3'}b^{3'}=\frac{1}{2}\overline{\nabla}_{3'}h_{\alpha\alpha}b^{3'}
\end{align*}
So 
\begin{align}
 & \int_{\Sigma_{n}}(H-H_{e})<v_{e},b>_{e}d\mu_{e}\nonumber \\
= & \int_{\Sigma_{n}}(-\frac{1}{2}h_{\alpha\beta}A_{\alpha\beta}b^{3'}+\frac{1}{2}h_{\beta3'}A_{\alpha\beta}b^{\alpha}\nonumber \\
 & -\frac{1}{2}\overline{\nabla}_{\alpha}h_{\alpha3'}b^{3'}+\frac{1}{2}\overline{\nabla}_{3'}h_{\alpha\alpha}b^{3'})d\mu_{e}.\label{Simplified expression}
\end{align}
We are going to prove that for any $\varepsilon>0$, there exists
$N>0$ such that when $n>N$
\[
|\int_{\Sigma_{n}}(H-H_{e})<v_{e},b>_{e}d\mu_{e}-\int_{\Sigma_{n}}(-\frac{1}{2}v^{l}\partial_{i}h_{il}+\frac{1}{2}v^{l}\partial_{l}h_{ii})d\mu_{e}|\leq\varepsilon.
\]
The outline of the proof is, for any $\varepsilon>0$, we can choose
$s>0$, sufficiently small and $K>0$ sufficiently large, and $N=N(s,K)$
such that the above relationship holds. 

First for any $s>0$ and $K>0$, we can find $N$ when $n>N$ 
\[
Kr_{0}(\Sigma_{n})<sH^{-1}(\Sigma_{n}).
\]
 So we can divide the integral into three parts, $\int_{\Sigma_{n}\cap B_{sH^{-1}}^{c}}$,
$\int_{\Sigma_{n}\cap B_{Kr_{0}}}$, $\int_{\Sigma_{n}\cap(B_{sH^{-1}}\setminus B_{Kr_{0}})}$.

For $r$ large, choose $\{\bar{x}_{i}\}$ (defined by (\ref{x-bar}))
as the coordinates on $M^{r}$. Define 
\[
h_{\bar{i}\bar{j}}^{r}(\bar{x})=rh_{ij}(r\bar{x}).
\]
Also we denote $h_{\bar{\alpha}\bar{\beta}}^{r}=rh_{\alpha\beta}(r\bar{x})$
and $\bar{3}'$ is used in the similar way. 

Let $\bar{b}=b$ and $\bar{v}$ to be the vector with $\bar{v}^{\bar{3}'}=1,\bar{v}^{\bar{\alpha}}=0,$
so $\bar{v}$ is the unit normal vector of the hypersurface in $\delta_{\bar{i}\bar{j}}$
metric. $ $ And $d\bar{\mu}_{e}$ represents the volume form of the
hypersurface in the metric $\delta_{\bar{i}\bar{j}}.$

We assume 
\[
|h_{ij}|+|x||h_{ij,k}|\leq C_{1}|x|^{-1}.
\]

Note that the scalar curvature $R_{g}$ is $L^{1}$ integrable and
$R_{g}=h_{ij,ij}-h_{ii,jj}+O(|x|^{-4}).$ So $ $$h_{ij,ij}-h_{ii,jj}\in L^{1}$
. Define
\[
F(r)=\int_{M\cap B_{r}^{c}(0)}|h_{ij,ij}-h_{ii,jj}|dvol(M).
\]
We have 
\[
\lim_{r\rightarrow\infty}F(r)=0.
\]

\begin{lem}\label{H^-1 scale}For any $\varepsilon>0$ and any small
$s>0$, we can choose $N_{1}=N_{1}(\varepsilon,s)$ such that when
$n>N_{1}$ 
\[
|\int_{\Sigma_{n}\cap B_{sH^{-1}}^{c}}(H-H_{e})<v_{e},b>_{e}d\mu_{e}-\int_{\Sigma_{n}\cap B_{sH^{-1}}^{c}}(-\frac{1}{2}v^{l}\partial_{i}h_{il}+\frac{1}{2}v^{l}\partial_{l}h_{ii})d\mu_{e}|\leq\frac{2}{9}\varepsilon+C(C_{1})s.
\]

\end{lem}

\begin{proof}Consider $\Sigma^{\frac{2}{H}}.$ From (\ref{Simplified expression}),
we have
\begin{align*}
 & \int_{\Sigma_{n}\cap B_{sH^{-1}}^{c}}(H-H_{e})<v_{e}\cdot b>_{e}d\mu_{e}\\
 & =\int_{\Sigma_{n}^{\frac{2}{H}}\cap B_{s/2}^{c}}(-\frac{1}{2}h_{\bar{\alpha}\bar{\beta}}^{2/H}A_{\bar{\alpha}\bar{\beta}}(M^{\frac{2}{H}})\bar{b}^{\bar{3}'}+\frac{1}{2}h_{\bar{\beta}\bar{3}'}^{2/H}A_{\bar{\alpha}\bar{\beta}}(M^{\frac{2}{H}})\bar{b}^{\bar{\alpha}}\\
 & -\frac{1}{2}\bar{\nabla}_{\bar{\alpha}}h_{\bar{\alpha}\bar{3'}}^{2/H}\bar{b}^{\bar{3}'}+\frac{1}{2}\bar{\nabla}_{\bar{3}'}h_{\bar{\alpha}\bar{\alpha}}^{2/H}\bar{b}^{\bar{3}'})d\bar{\mu}_{e}(\Sigma_{n}^{\frac{2}{H}}).
\end{align*}
From Lemma \ref{Blow down by H/2}, for fixed $s>0$, as $n\rightarrow\infty$,
$\Sigma_{n}^{\frac{2}{H}}\cap B_{s/2}^{c}$ will converge in $C^{2,\alpha}$
sense to $S_{1}(a)\cap B_{s/2}^{c}$. So we have $A_{\bar{\alpha}\bar{\beta}}(\Sigma_{n}^{\frac{2}{H}})-f_{\bar{\alpha}\bar{\beta}}(\Sigma_{n}^{\frac{2}{H}})\rightarrow0$
in $C^{\alpha}$ sense and $\bar{v}^{i}\rightarrow\bar{x}^{i}-a^{i}$.
We know on $\Sigma_{n}\cap B_{sH^{-1}}^{c}$
\[
|h_{\bar{i}\bar{j}}^{r}(\bar{x})|_{C^{0}}\leq Cs^{-1}H,|h_{\bar{i}\bar{j},\bar{k}}^{r}(\bar{x})|_{C^{0}}\leq Cs^{-2}H^{2}.
\]
So we know there is $N_{1}'>0$, such that when $n>N_{1}'$ 
\begin{align*}
 & |\int_{\Sigma_{n}^{\frac{2}{H}}\cap B_{s/2}^{c}}-\frac{1}{2}h_{\bar{\alpha}\bar{\beta}}^{2/H}A_{\bar{\alpha}\bar{\beta}}(\Sigma_{n}^{\frac{2}{H}})\bar{b}^{\bar{3}'}+\frac{1}{2}h_{\bar{\beta}\bar{3}'}^{2/H}A_{\bar{\alpha}\bar{\beta}}(\Sigma_{n}^{\frac{2}{H}})\bar{b}^{\bar{\alpha}}\\
 & -\frac{1}{2}\bar{\nabla}_{\bar{\alpha}}h_{\bar{\alpha}\bar{3'}}^{2/H}\bar{b}^{\bar{3}'}+\frac{1}{2}\bar{\nabla}_{\bar{3}'}h_{\bar{\alpha}\bar{\alpha}}^{2/H}\bar{b}^{\bar{3}'}d\bar{\mu}_{e}(\Sigma_{n}^{\frac{2}{H}})\\
 & -\int_{\Sigma_{n}^{\frac{2}{H}}\cap B_{s/2}^{c}}-\frac{1}{2}h_{\bar{\alpha}\bar{\alpha}}^{2/H}\bar{b}^{\bar{3}'}+\frac{1}{2}h_{\bar{\alpha}\bar{3}'}^{2/H}\bar{b}^{\bar{\alpha}}\\
 & -\frac{1}{2}\bar{\nabla}_{\bar{\alpha}}h_{\bar{\alpha}\bar{3'}}^{2/H}\bar{b}^{\bar{3}'}+\frac{1}{2}\bar{\nabla}_{\bar{3}'}h_{\bar{\alpha}\bar{\alpha}}^{2/H}\bar{b}^{\bar{3}'}d\bar{\mu}_{e}(\Sigma_{n}^{\frac{2}{H}})|\\
\leq & \frac{\varepsilon}{9}.
\end{align*}
It is important to note that
\begin{align*}
 & \int_{\Sigma_{n}^{\frac{2}{H}}\cap B_{s/2}^{c}}-\frac{1}{2}h_{\bar{\alpha}\bar{\alpha}}^{2/H}\bar{b}^{\bar{3}'}+\frac{1}{2}h_{\bar{\alpha}\bar{3}'}^{2/H}\bar{b}^{\bar{\alpha}}-\frac{1}{2}\bar{\nabla}_{\bar{\alpha}}h_{\bar{\alpha}\bar{3'}}^{2/H}\bar{b}^{\bar{3}'}+\frac{1}{2}\bar{\nabla}_{\bar{3}'}h_{\bar{\alpha}\bar{\alpha}}^{2/H}\bar{b}^{\bar{3}'}\\
= & \int_{\Sigma_{n}^{\frac{2}{H}}\cap B_{s/2}^{c}}(-\frac{1}{2}(h_{\bar{\alpha}\bar{\alpha}}^{2/H}+h_{\bar{3}'\bar{3}'}^{2/H})\bar{b}^{\bar{3}'}+\frac{1}{2}h_{\bar{\alpha}\bar{3}'}^{2/H}\bar{b}^{\bar{\alpha}}+\frac{1}{2}h_{\bar{3}'\bar{3}'}^{2/H}\bar{b}^{\bar{3}'}\\
 & -\frac{1}{2}(\bar{\nabla}_{\bar{\alpha}}h_{\bar{\alpha}\bar{3'}}^{2/H}+\bar{\nabla}_{\bar{3}'}h_{\bar{3}'\bar{3}'}^{2/H})\bar{b}^{\bar{3}'}+\frac{1}{2}\bar{\nabla}_{\bar{3}'}h_{\bar{\alpha}\bar{\alpha}}^{2/H}\bar{b}^{\bar{3}'}+\frac{1}{2}\bar{\nabla}_{\bar{3}'}h_{\bar{3}'\bar{3}'}^{2/H}\bar{b}^{\bar{3}'})d\bar{\mu}_{e}(\Sigma_{n}^{\frac{2}{H}})\\
= & \int_{\Sigma_{n}^{\frac{2}{H}}\cap B_{s/2}^{c}}(-\frac{1}{2}h_{\bar{i}\bar{j}}^{2/H}\bar{v}^{\bar{m}}\bar{b}^{\bar{m}}+\frac{1}{2}h_{\bar{i}\bar{l}}^{2/H}\bar{v}^{\bar{l}}\bar{b}^{\bar{i}}\\
 & -\frac{1}{2}\bar{v}^{\bar{l}}\bar{\nabla}_{\bar{i}}h_{\bar{i}\bar{l}}^{2/H}\bar{v}^{\bar{m}}\bar{b}^{\bar{m}}+\frac{1}{2}v^{\bar{l}}\bar{\nabla}_{\bar{l}}h_{\bar{i}\bar{i}}^{2/H}\bar{v}^{\bar{m}}\bar{b}^{\bar{m}})d\bar{\mu}_{e}(\Sigma_{n}^{\frac{2}{H}})\\
= & \int_{\Sigma_{n}^{\frac{2}{H}}\cap B_{s/2}^{c}}(-\frac{1}{2}h_{\bar{i}\bar{j}}^{2/H}\bar{v}^{\bar{m}}\bar{b}^{\bar{m}}+\frac{1}{2}h_{\bar{i}\bar{l}}^{2/H}\bar{v}^{\bar{l}}\bar{b}^{\bar{i}}\\
 & -\frac{1}{2}\bar{v}^{\bar{l}}\partial_{\bar{i}}h_{\bar{i}\bar{l}}^{2/H}(\bar{x}^{\bar{m}}-a^{\bar{m}})\bar{b}^{\bar{m}}+\frac{1}{2}v^{\bar{l}}\partial_{\bar{l}}h_{\bar{i}\bar{i}}^{2/H}(\bar{x}^{\bar{m}}-a^{\bar{m}})\bar{b}^{\bar{m}})d\bar{\mu}_{e}(\Sigma_{n}^{\frac{2}{H}})+o(1)\\
= & \int_{\Sigma_{n}^{\frac{2}{H}}\cap B_{s/2}^{c}}(-\frac{1}{2}h_{\bar{i}\bar{j}}^{2/H}\bar{v}^{\bar{m}}\bar{b}^{\bar{m}}+\frac{1}{2}h_{\bar{i}\bar{l}}^{2/H}\bar{v}^{\bar{l}}\bar{b}^{\bar{i}}\\
 & -\frac{1}{2}\bar{v}^{\bar{l}}\partial_{\bar{i}}h_{\bar{i}\bar{l}}^{2/H}\bar{x}^{\bar{m}}\bar{b}^{\bar{m}}+\frac{1}{2}v^{\bar{l}}\partial_{\bar{l}}h_{\bar{i}\bar{i}}^{2/H}\bar{x}^{\bar{m}}\bar{b}^{\bar{m}})d\bar{\mu}_{e}(\Sigma_{n}^{\frac{2}{H}})\\
 & -\int_{\Sigma_{n}^{\frac{2}{H}}\cap B_{s/2}^{c}}(\frac{1}{2}\bar{v}^{\bar{l}}\partial_{\bar{i}}h_{\bar{i}\bar{l}}^{2/H}-\frac{1}{2}v^{\bar{l}}\partial_{\bar{l}}h_{\bar{i}\bar{i}}^{2/H})d\bar{\mu}_{e}(\Sigma_{n}^{\frac{2}{H}})+o(1).
\end{align*}
This $o(1)$ means $\lim_{n\rightarrow\infty}o(1)=0.$ $\bar{\nabla}_{i}h_{jk}$
can be replaced by $\partial_{i}h_{jk}$ because the difference of
the two are high order terms.

We denote the part of $M^{r}$ between $\partial\Phi_{r}(K')$ and
$\partial\Sigma_{n}^{\frac{2}{H}}$ as $int(\Sigma_{n}^{\frac{2}{H}})$.
Then by divergence formula we have
\begin{align*}
\lefteqn{} & \int_{\Sigma_{n}^{\frac{2}{H}}\cap B_{s/2}^{c}}(-\frac{1}{2}h_{\bar{i}\bar{i}}^{2/H}\bar{v}^{\bar{m}}\bar{b}^{\bar{m}}+\frac{1}{2}h_{\bar{i}\bar{l}}^{2/H}\bar{v}^{\bar{l}}b^{\bar{i}}\\
 & -\frac{1}{2}\bar{v}^{\bar{l}}\partial_{\bar{i}}h_{\bar{i}\bar{l}}^{2/H}\bar{x}^{\bar{m}}\bar{b}^{\bar{m}}+\frac{1}{2}\bar{v}^{\bar{l}}\partial_{\bar{l}}h_{\bar{i}\bar{i}}^{2/H}\bar{x}^{\bar{m}}\bar{b}^{\bar{m}})d\bar{\mu}_{e}(\Sigma_{n}^{\frac{2}{H}})\\
 & =\int_{int(\Sigma_{n}^{\frac{2}{H}})\cap\partial B_{s/2}^{c}(0)}(-\frac{1}{2}h_{\bar{i}\bar{i}}^{2/H}\bar{v}^{\bar{m}}\bar{b}^{\bar{m}}+\frac{1}{2}h_{\bar{i}\bar{l}}^{2/H}\bar{v}^{\bar{l}}b^{\bar{i}}\\
 & -\frac{1}{2}\bar{v}^{\bar{l}}\partial_{\bar{i}}h_{\bar{i}\bar{l}}^{2/H}\bar{x}^{\bar{m}}\bar{b}^{\bar{m}}+\frac{1}{2}\bar{v}^{\bar{l}}\partial_{\bar{l}}h_{\bar{i}\bar{i}}^{2/H}\bar{x}^{\bar{m}}\bar{b}^{\bar{m}})d\bar{\mu}_{e}(\Sigma_{n}^{\frac{2}{H}})\\
 & -\frac{1}{2}\int_{int(\Sigma_{n}^{\frac{2}{H}})\cap B_{s/2}^{c}(0)}(h_{\bar{i}\bar{l},\bar{i}\bar{l}}^{2/H}-h_{\bar{i}\bar{i},\bar{l}\bar{l}}^{2/H})(\bar{x}^{\bar{m}}\bar{b}^{\bar{m}})d\bar{vol}_{e}.
\end{align*}
We know 
\begin{align*}
 & |\int_{int(\Sigma_{n}^{\frac{2}{H}})\cap\partial B_{s/2}(0)}(-\frac{1}{2}h_{\bar{i}\bar{i}}^{2/H}\bar{v}^{\bar{m}}\bar{b}^{\bar{m}}+\frac{1}{2}h_{\bar{i}\bar{l}}^{2/H}\bar{v}^{\bar{l}}b^{\bar{i}}\\
 & -\frac{1}{2}\bar{v}^{\bar{l}}\partial_{\bar{i}}h_{\bar{i}\bar{l}}^{2/H}\bar{x}^{\bar{m}}\bar{b}^{\bar{m}}+\frac{1}{2}\bar{v}^{\bar{l}}\partial_{\bar{l}}h_{\bar{i}\bar{i}}^{2/H}\bar{x}^{\bar{m}}\bar{b}^{\bar{m}})d\bar{\mu}_{e}(\tilde{\Sigma}_{n})|\\
 & \leq C(C_{1})s
\end{align*}
 
\begin{align*}
 & \lefteqn{}|\int_{int(\Sigma_{n}^{\frac{2}{H}})\cap B_{s/2}^{c}(0)}(h_{\bar{i}\bar{l},\bar{i}\bar{l}}^{2/H}-h_{\bar{i}\bar{i},\bar{l}\bar{l}}^{2/H})(\bar{x}^{\bar{m}}\bar{b}^{\bar{m}})d\bar{vol}_{e}|\leq C|\int_{int(\tilde{\Sigma}_{n})\cap B_{s/2}^{c}(0)}|h_{\bar{i}\bar{l},\bar{i}\bar{l}}^{2/H}-h_{\bar{i}\bar{i},\bar{l}\bar{l}}^{2/H}|d\bar{vol}_{e}\\
 & =C\int_{int(\Sigma_{n})\cap B_{sH^{-1}}^{c}(0)}|h_{il,il}-h_{ii,ll}|dvol_{e}\\
 & \leq CF(sH^{-1}).
\end{align*}
And 
\begin{align*}
 & \int_{\Sigma_{n}^{\frac{2}{H}}\cap B_{s/2}^{c}}(\frac{1}{2}\bar{v}^{\bar{l}}\partial_{\bar{i}}h_{\bar{i}\bar{l}}^{2/H}-\frac{1}{2}v^{\bar{l}}\partial_{\bar{l}}h_{\bar{i}\bar{i}}^{2/H})d\bar{\mu}_{e}(\Sigma_{n}^{\frac{2}{H}})\\
= & \int_{\Sigma_{n}\cap B_{sH^{-1}}^{c}}(\frac{1}{2}v^{l}\partial_{i}h_{il}-\frac{1}{2}v^{l}\partial_{l}h_{ii})d\mu_{e}
\end{align*}

So we know 
\begin{align*}
 & |\int_{\Sigma_{n}\cap B_{sH^{-1}}^{c}}(H-H_{e})<v_{e},b>_{e}d\mu_{e}\\
 & -\int_{\Sigma_{n}\cap B_{sH^{-1}}^{c}}(-\frac{1}{2}v^{l}\partial_{i}h_{il}+\frac{1}{2}v^{l}\partial_{l}h_{ii})d\mu_{e}|\\
\leq & \frac{\varepsilon}{9}+o(1)+C(C_{1})s+CF(sH^{-1}).
\end{align*}
From
\[
\lim_{n\rightarrow\infty}F(sH^{-1})=0
\]
So we can choose $N_{1}=N_{1}(\varepsilon,s)$ such that the lemma
holds. 

\end{proof}

\begin{lem}For any $K>0$, there is $N_{2}=N_{2}(\varepsilon,K)>0$,
such that when $n>N_{2},$
\[
|\int_{\Sigma_{n}\cap B_{Kr_{0}}}(H-H_{e})<v_{e},b>_{e}d\mu_{e}-\int_{\Sigma_{n}\cap B_{Kr_{0}}}(-\frac{1}{2}v^{l}\partial_{i}h_{il}+\frac{1}{2}v^{l}\partial_{l}h_{ii})d\mu_{e}|\leq\frac{\varepsilon}{3}.
\]

\end{lem}

\begin{proof}Consider $\Sigma^{r_{0}}.$ From (\ref{Simplified expression})
we have 
\begin{align*}
 & \int_{\Sigma_{n}\cap B_{Kr_{0}}}(H-H_{e})<v_{e}\cdot b>_{e}d\mu_{e}\\
 & =\int_{\Sigma_{n}^{r_{0}}\cap B_{K}}(-\frac{1}{2}h_{\bar{\alpha}\bar{\beta}}^{r_{0}}A_{\bar{\alpha}\bar{\beta}}(M^{r_{0}})\bar{b}^{\bar{3}'}+\frac{1}{2}h_{\bar{\beta}\bar{3}'}^{r_{0}}A_{\bar{\alpha}\bar{\beta}}(M^{r_{0}})\bar{b}^{\bar{\alpha}}\\
 & -\frac{1}{2}\bar{\nabla}_{\bar{\alpha}}h_{\bar{\alpha}\bar{3'}}^{r_{0}}\bar{b}^{\bar{3}'}+\frac{1}{2}\bar{\nabla}_{\bar{3}'}h_{\bar{\alpha}\bar{\alpha}}^{r_{0}}\bar{b}^{\bar{3}'})d\bar{\mu}_{e}(\Sigma_{n}^{\frac{2}{H}}).
\end{align*}
From Lemma \ref{Blow down by r_0^{-1}} and Lemma \ref{no-neck lemma},
$A_{\bar{\alpha}\bar{\beta}}(\Sigma_{n}^{r_{0}})\rightarrow0$ and
$\bar{v}^{\bar{m}}\rightarrow b^{\bar{m}}.$ So as $n\rightarrow\infty$,
we have $\bar{b}^{\bar{3}'}\rightarrow1$ and from
\[
|h_{\bar{i}\bar{j}}^{r_{0}}|+|\bar{x}|^{-1}|h_{\bar{i}\bar{j},\bar{k}}^{r_{0}}|\leq C_{1}|\bar{x}|^{-1}
\]
we have 
\[
\int_{\Sigma_{n}^{r_{0}}\cap B_{K}}-\frac{1}{2}h_{\bar{\alpha}\bar{\beta}}^{r_{0}}A_{\bar{\alpha}\bar{\beta}}(\Sigma_{n}^{r_{0}})\bar{b}^{\bar{3}'}+\frac{1}{2}h_{\bar{\beta}\bar{3}'}^{r_{0}}A_{\bar{\alpha}\bar{\beta}}(\Sigma_{n}^{r_{0}})\bar{b}^{\bar{\alpha}}\rightarrow0
\]
and 
\begin{align*}
 & \int_{\Sigma_{n}^{r_{0}}\cap B_{K}}(-\frac{1}{2}\bar{\nabla}_{\bar{\alpha}}h_{\bar{\alpha}\bar{3'}}^{r_{0}}\bar{b}^{\bar{3}'}+\frac{1}{2}\bar{\nabla}_{\bar{3}'}h_{\bar{\alpha}\bar{\alpha}}^{r_{0}}\bar{b}^{\bar{3}'})d\bar{\mu}_{e}(\Sigma_{n}^{r_{0}})\\
= & \int_{\Sigma_{n}^{r_{0}}\cap B_{K}}(-\frac{1}{2}\bar{\nabla}_{\bar{\alpha}}h_{\bar{\alpha}\bar{3'}}^{r_{0}}\bar{b}^{\bar{3}'}-\frac{1}{2}\bar{\nabla}_{\bar{3}}h_{\bar{3}\bar{3'}}^{r_{0}}\bar{b}^{\bar{3}'})\\
 & +(\frac{1}{2}\bar{\nabla}_{\bar{3}'}h_{\bar{\alpha}\bar{\alpha}}^{r_{0}}\bar{b}^{\bar{3}'}+\frac{1}{2}\bar{\nabla}_{\bar{3}}h_{\bar{3}\bar{3'}}^{r_{0}}\bar{b}^{\bar{3}'})d\bar{\mu}_{e}(\Sigma_{n}^{r_{0}})\\
= & \int_{\Sigma_{n}^{r_{0}}\cap B_{K}}(-\frac{1}{2}\bar{v}^{\bar{l}}\partial_{\bar{i}}h_{\bar{i}\bar{l}}^{r_{0}}+\frac{1}{2}\bar{v}^{\bar{l}}\partial_{\bar{l}}h_{\bar{i}\bar{i}}^{r_{0}})d\bar{\mu}_{e}(\Sigma_{n}^{r_{0}})+o(1)\\
= & \int_{\Sigma_{n}\cap B_{Kr_{0}}}(-\frac{1}{2}v^{l}\partial_{i}h_{il}+\frac{1}{2}v^{l}\partial_{l}h_{ii})d\mu_{e}+o(1).
\end{align*}
So
\begin{align*}
 & |\int_{\Sigma_{n}\cap B_{Kr_{0}}}(H-H_{e})<v_{e},b>_{e}d\mu_{e}\\
 & -\int_{\Sigma_{n}\cap B_{Kr_{0}}}(-\frac{1}{2}v^{l}\partial_{i}h_{il}+\frac{1}{2}v^{l}\partial_{l}h_{ii})d\mu_{e}|.\\
\leq & o(1).
\end{align*}
So we can choose $N_{2}$ such that the lemma holds. 

\end{proof}

\begin{lem}We can choose a small $s>0$ and a large $K>0$ and $N_{3}>0$
such that when $n>N_{3}$,
\begin{align*}
 & |\int_{\Sigma_{n}\cap(B_{sH^{-1}}\setminus B_{Kr_{0}})}(H-H_{e})<v_{e}\cdot b>_{e}d\mu_{e}\\
 & -\int_{\Sigma_{n}\cap(B_{sH^{-1}}\setminus B_{Kr_{0}})}(-\frac{1}{2}v^{l}\partial_{i}h_{il}v^{m}b^{m}+\frac{1}{2}v^{l}\partial_{l}h_{ii}v^{m}b^{m})d\mu_{e}|\\
\leq & \frac{\varepsilon}{3}
\end{align*}

\end{lem}

\begin{proof}
\begin{align*}
 & \int_{\Sigma_{n}\cap(B_{sH^{-1}}\setminus B_{Kr_{0}})}(H-H_{e})<v_{e}\cdot b>_{e}d\mu_{e}\\
= & \int_{\Sigma_{n}\cap(B_{sH^{-1}}\setminus B_{Kr_{0}})}(-\frac{1}{2}h_{\alpha\beta}A_{\alpha\beta}b^{3}+\frac{1}{2}h_{\beta3}A_{\alpha\beta}b^{3}\\
 & -\frac{1}{2}\bar{\nabla}_{\alpha}h_{\alpha3}b^{3}+\frac{1}{2}\bar{\nabla}_{3}h_{\alpha\alpha}b^{3})d\mu_{e}
\end{align*}
From Lemma \ref{The new estimate for the second fundamental form}
we have
\begin{align*}
 & |\int_{\Sigma_{n}\cap(B_{sH^{-1}}\setminus B_{Kr_{0}})}-\frac{1}{2}h_{\alpha\beta}A_{\alpha\beta}b^{3}+\frac{1}{2}h_{\beta3}A_{\alpha\beta}b^{3}d\mu_{e}|\\
\leq & \sum_{i=1}^{l_{n}}\int_{\Sigma_{n}\cap(B_{Kr_{0}e^{iL}}\backslash B_{Kr_{0}e^{(i-1)L}})}C|x|^{-2}(e^{-\frac{i}{2}L}+e^{-\frac{l_{n}-i}{2}L})(r_{0}^{-\frac{1}{2}}+s)^{\frac{1}{2}}d\mu_{e}\\
\leq & C(r_{0}^{-\frac{1}{2}}+s)^{\frac{1}{2}}.
\end{align*}
For the second part, we have
\begin{align*}
 & \int_{\Sigma_{n}\cap(B_{sH^{-1}}\setminus B_{Kr_{0}})}-\frac{1}{2}\bar{\nabla}_{\alpha}h_{\alpha3}b^{3}+\frac{1}{2}\bar{\nabla}_{3}h_{\alpha\alpha}b^{3}d\mu_{e}\\
= & \int_{\Sigma_{n}\cap(B_{sH^{-1}}\setminus B_{Kr_{0}})}-\frac{1}{2}v^{l}\partial_{i}h_{il}v^{m}b^{m}+\frac{1}{2}v^{l}\partial_{l}h_{ii}v^{m}b^{m}d\mu_{e}+o(1)
\end{align*}
For each $n$ we can choose $p_{n}\in\Sigma_{n}\cap(B_{Kr_{0}e^{(\frac{l_{n}}{2}+1)L}}\backslash B_{Kr_{0}e^{\frac{l_{n}}{2}L}})$
such that for $v_{n}=v(p_{n})$ Corollary \ref{Choose a mean value of the normal vector}
holds. So we have
\begin{align*}
 & \int_{\Sigma_{n}\cap(B_{sH^{-1}}\setminus B_{Kr_{0}})}(-\frac{1}{2}v^{l}\partial_{i}h_{il}v^{m}b^{m}+\frac{1}{2}v^{l}\partial_{l}h_{ii}v^{m}b^{m})d\mu_{e}\\
= & \int_{\Sigma_{n}\cap(B_{sH^{-1}}\setminus B_{Kr_{0}})}-\frac{1}{2}v^{l}\partial_{i}h_{il}(v^{m}-v_{n}^{m})b^{m}+\frac{1}{2}v^{l}\partial_{l}h_{ii}(v^{m}-v_{n}^{m})b^{m}d\mu_{e}\\
 & +(v_{n}^{m}b^{m})\int_{\Sigma_{n}\cap(B_{sH^{-1}}\setminus B_{Kr_{0}})}(-\frac{1}{2}v^{l}\partial_{i}h_{il}+\frac{1}{2}v^{l}\partial_{l}h_{ii})d\mu_{e}.
\end{align*}
For the first term on the right hand side, we have:
\begin{align*}
\lefteqn{|\int_{\Sigma_{n}\cap(B_{sH^{-1}}\setminus B_{Kr_{0}})}-\frac{1}{2}v^{l}\partial_{i}h_{il}(v^{m}-v_{n}^{m})b^{m}+\frac{1}{2}v^{l}\partial_{l}h_{ii}(v^{m}-v_{n}^{m})b^{m}d\mu_{e}|}\\
 & \leq\sum_{i=1}^{l_{n}}|\int_{\Sigma_{n}\cap(B_{Kr_{0}e^{iL}}\setminus B_{Kr_{0}e^{(i-1)L}})}-\frac{1}{2}v^{l}\partial_{i}h_{il}(v^{m}-v_{n}^{m})b^{m}+\frac{1}{2}v^{l}\partial_{l}h_{ii}(v^{m}-v_{n}^{m})b^{m}d\mu_{e}|\\
 & \leq\sum_{i=1}^{l_{n}/2}C(e^{-\frac{1}{2}iL}+e^{-\frac{1}{4}l_{n}L})(s+r_{0}^{-\frac{1}{2}})^{\frac{1}{2}}+\sum_{i=l_{n/2}+1}^{l_{n}}C(e^{-\frac{1}{4}l_{n}L}+e^{-\frac{1}{2}(l_{n}-i)L})(s+r_{0}^{-\frac{1}{2}})^{\frac{1}{2}}\\
 & \leq C(s+r_{0}^{-\frac{1}{2}})^{\frac{1}{2}}.
\end{align*}
For the second term, recall Lemma \ref{no-neck lemma}. We can choose
$s$ small, $K$ large and $n$ large such that $|v_{n}-b|\leq\frac{\varepsilon}{9\tilde{C}}$
. So as long as 
\[
\int_{\Sigma_{n}\cap(B_{sH^{-1}}\setminus B_{Kr_{0}})}(-\frac{1}{2}v^{l}\partial_{i}h_{il}+\frac{1}{2}v^{l}\partial_{l}h_{ii})d\mu_{e}
\]
 is bounded by $\tilde{C}>0$ (which will be verified by the following
Lemma \ref{a bound on the intermediate part}),
\[
(v_{n}^{m}b^{m}-1)\int_{\Sigma_{n}\cap(B_{sH^{-1}}\setminus B_{Kr_{0}})}(-\frac{1}{2}v^{l}\partial_{i}h_{il}+\frac{1}{2}v^{l}\partial_{l}h_{ii})d\mu_{e}
\]
is bounded by $\frac{\varepsilon}{9}$. Then we have
\begin{align*}
 & |\int_{\Sigma_{n}\cap(B_{sH^{-1}}\setminus B_{Kr_{0}})}(H-H_{e})<v_{e}\cdot b>_{e}d\mu_{e}\\
 & -\int_{\Sigma_{n}\cap(B_{sH^{-1}}\setminus B_{Kr_{0}})}(-\frac{1}{2}v^{l}\partial_{i}h_{il}v^{m}b^{m}+\frac{1}{2}v^{l}\partial_{l}h_{ii}v^{m}b^{m})d\mu_{e}|\\
\leq & C(s+r_{0}^{-\frac{1}{2}})^{\frac{1}{2}}+o(1)+\frac{\varepsilon}{9}.
\end{align*}
So one can find $N_{3}$ such that when $n>N_{3}$, the lemma holds.
The lemma then follows from Lemma \ref{a bound on the intermediate part}.

\end{proof}

\begin{lem}\label{a bound on the intermediate part} There exists
$\tilde{C}(m)>0$ such that 
\[
|\int_{\Sigma_{n}\cap(B_{sH^{-1}}\setminus B_{Kr_{0}})}(-\frac{1}{2}v^{l}\partial_{i}h_{il}+\frac{1}{2}v^{l}\partial_{l}h_{ii})d\mu_{e}|\leq\tilde{C}.
\]

\end{lem}

\begin{proof} Note that 
\begin{align}
 & |\int_{\Sigma_{n}}(-\frac{1}{2}v^{l}\partial_{i}h_{il}+\frac{1}{2}v^{l}\partial_{l}h_{ii})d\mu_{e}-\int_{\partial B_{r_{0}}}(-\frac{1}{2}v^{l}\partial_{i}h_{il}+\frac{1}{2}v^{l}\partial_{l}h_{ii})d\mu_{e}|\nonumber \\
\leq & \frac{1}{2}\int_{int(\Sigma_{n})\backslash B_{r_{0}}}|h_{il,il}-h_{ii,ll}|d\bar{vol}_{e}\nonumber \\
\leq & CF(r_{0})\rightarrow0.\label{mass relation}
\end{align}
And 
\[
\int_{\partial B_{r_{0}}}(-\frac{1}{2}v^{l}\partial_{i}h_{il}+\frac{1}{2}v^{l}\partial_{l}h_{ii})d\mu_{e}\rightarrow-16\pi m.
\]
\begin{align*}
 & \int_{\Sigma_{n}}(-\frac{1}{2}v^{l}\partial_{i}h_{il}+\frac{1}{2}v^{l}\partial_{l}h_{ii})d\mu_{e}\\
= & \int_{\Sigma_{n}\cap B_{sH^{-1}}^{c}}+\int_{\Sigma_{n}\cap B_{Kr_{0}}}+\int_{\Sigma_{n}\cap(B_{sH^{-1}}\setminus B_{Kr_{0}})}(-\frac{1}{2}v^{l}\partial_{i}h_{il}+\frac{1}{2}v^{l}\partial_{l}h_{ii})d\mu_{e}.
\end{align*}
So the left is to prove that 
\[
\int_{\Sigma_{n}\cap B_{sH^{-1}}^{c}}(-\frac{1}{2}v^{l}\partial_{i}h_{il}+\frac{1}{2}v^{l}\partial_{l}h_{ii})d\mu_{e}
\]
 and 
\[
\int_{\Sigma_{n}\cap B_{Kr_{0}}}(-\frac{1}{2}v^{l}\partial_{i}h_{il}+\frac{1}{2}v^{l}\partial_{l}h_{ii})d\mu_{e}
\]
 are bounded.

For the first one 
\begin{align*}
 & \int_{\Sigma_{n}\cap B_{sH^{-1}}^{c}}(\frac{1}{2}v^{l}\partial_{i}h_{il}-\frac{1}{2}v^{l}\partial_{l}h_{ii})d\mu_{e}\\
= & \int_{\Sigma_{n}^{\frac{2}{H}}\cap B_{\frac{s}{2}}^{c}}(\frac{1}{2}\bar{v}^{\bar{l}}\partial_{\bar{i}}h_{\bar{i}\bar{l}}^{2/H}-\frac{1}{2}v^{\bar{l}}\partial_{\bar{l}}h_{\bar{i}\bar{i}}^{2/H})d\bar{\mu}_{e}(\Sigma_{n}^{\frac{2}{H}})\\
= & \int_{\partial B_{\frac{s}{2}}\cap int(\Sigma_{n}^{\frac{2}{H}})}(\frac{1}{2}\bar{v}^{\bar{l}}\partial_{\bar{i}}h_{\bar{i}\bar{l}}^{2/H}-\frac{1}{2}v^{\bar{l}}\partial_{\bar{l}}h_{\bar{i}\bar{i}}^{2/H})d\bar{\mu}_{e}\\
 & +\int_{int(\Sigma_{n}^{\frac{2}{H}})\cap B_{\frac{s}{2}}^{c}}(h_{\bar{i}\bar{l},\bar{i}\bar{l}}^{2/H}-h_{\bar{i}\bar{i},\bar{l}\bar{l}}^{2/H})d\bar{vol}_{e}.
\end{align*}
 From 
\begin{align*}
 & |\int_{\partial B_{\frac{s}{2}}\cap int(\Sigma_{n}^{\frac{2}{H}})}(\frac{1}{2}\bar{v}^{\bar{l}}\partial_{\bar{i}}h_{\bar{i}\bar{l}}^{2/H}-\frac{1}{2}v^{\bar{l}}\partial_{\bar{l}}h_{\bar{i}\bar{i}}^{2/H})d\bar{\mu}_{e}|\\
\leq & \frac{1}{2}C_{1}|\frac{s}{2}|^{-2}|\partial B_{\frac{s}{2}}|\leq C(C_{1})
\end{align*}
and 
\begin{align*}
 & |\int_{int(\Sigma_{n}^{\frac{2}{H}})\cap B_{\frac{s}{2}}^{c}}(h_{\bar{i}\bar{l},\bar{i}\bar{l}}^{2/H}-h_{\bar{i}\bar{i},\bar{l}\bar{l}}^{2/H})d\bar{vol}_{e}|\\
= & |\int_{int(\Sigma_{n})\cap B_{sH^{-1}}^{c}}(h_{il,il}-h_{ii,ll})dvol_{e}|\\
\leq & F(sH^{-1}).
\end{align*}
So we have 
\[
|\int_{\Sigma_{n}\cap B_{sH^{-1}}^{c}}(\frac{1}{2}v^{l}\partial_{i}h_{il}-\frac{1}{2}v^{l}\partial_{l}h_{ii})d\mu_{e}|\leq C(C_{1})+F(sH^{-1}).
\]
For the second one 
\begin{align*}
 & \int_{\Sigma_{n}\cap B_{Kr_{0}}}(\frac{1}{2}v^{l}\partial_{i}h_{il}-\frac{1}{2}v^{l}\partial_{l}h_{ii})d\mu_{e}\\
= & \int_{\Sigma_{n}^{r_{0}}\cap B_{K}}(\frac{1}{2}\bar{v}^{\bar{l}}\partial_{\bar{i}}h_{\bar{i}\bar{l}}^{r_{0}}-\frac{1}{2}\bar{v}^{\bar{l}}\partial_{\bar{l}}h_{\bar{i}\bar{i}}^{r_{0}})d\bar{\mu}_{e}(\Sigma_{n}^{r_{0}})\\
= & \int_{\partial B_{K}\backslash int(\Sigma_{n}^{r_{0}})}(\frac{1}{2}\bar{v}^{\bar{l}}\partial_{\bar{i}}h_{\bar{i}\bar{l}}^{r_{0}}-\frac{1}{2}\bar{v}^{\bar{l}}\partial_{\bar{l}}h_{\bar{i}\bar{i}}^{r_{0}})d\bar{\mu}_{e}\\
 & +\int_{B_{K}\backslash int(\Sigma_{n}^{r_{0}})}\frac{1}{2}(h_{\bar{i}\bar{l},\bar{i}\bar{l}}^{r_{0}}-h_{\bar{i}\bar{i},\bar{l}\bar{l}}^{r_{0}})d\bar{vol}_{e}.
\end{align*}
 Note that
\begin{align*}
|\int_{\partial B_{K}\backslash int(\Sigma_{n}^{r_{0}})}(\frac{1}{2}\bar{v}^{\bar{l}}\partial_{\bar{i}}h_{\bar{i}\bar{l}}^{r_{0}}-\frac{1}{2}\bar{v}^{\bar{l}}\partial_{\bar{l}}h_{\bar{i}\bar{i}}^{r_{0}})d\bar{\mu}_{e}| & \leq C_{1}K^{-2}|\partial B_{K}|\le C(C_{1}),
\end{align*}
and 
\begin{align*}
|\int_{B_{K}\backslash int(\Sigma_{n}^{r_{0}})}\frac{1}{2}(h_{\bar{i}\bar{l},\bar{i}\bar{l}}^{r_{0}}-h_{\bar{i}\bar{i},\bar{l}\bar{l}}^{r_{0}})d\bar{vol}_{e}| & \leq F(r_{0}).
\end{align*}
So we have 
\[
|\int_{\Sigma_{n}\cap B_{Kr_{0}}}(\frac{1}{2}v^{l}\partial_{i}h_{il}-\frac{1}{2}v^{l}\partial_{l}h_{ii})d\mu_{e}|\leq C(C_{1})+F(r_{0}).
\]

\end{proof}

So for any $\varepsilon>0$ we can choose $s$ small ($s<\frac{\varepsilon}{9C(C_{1})}$)
and $K$ large and $N>\max\{N_{1},N_{2},N_{3}\}$ such that when $n>N$
we have
\[
|\int_{\Sigma_{n}}(H-H_{e})<v_{e},b>_{e}d\mu_{e}-\int_{\Sigma_{n}}(-\frac{1}{2}v^{l}\partial_{i}h_{il}+\frac{1}{2}v^{l}\partial_{l}h_{ii})d\mu_{e}|\leq\varepsilon.
\]
We choose $\varepsilon<\frac{|m|}{2}$. Pay attention to (\ref{mass relation}).
So when $r_{0}$ is sufficiently large 
\[
\int_{\Sigma_{n}}(H-H_{e})<v_{e},b>_{e}d\mu_{e}
\]
 cannot be $0$ which is a contradiction with (\ref{trivial identity}).
So there is $C>0$ such that $r_{1}\leq Cr_{0}.$ 

Now we prove that for a sequence of stable constant mean curvature
$\Sigma_{n}$ which separate $K'$ from infinity, if 
\[
\lim_{n\rightarrow\infty}r_{0}(\Sigma_{n})=\infty
\]
we have 
\[
\lim_{n\rightarrow\infty}\frac{r_{0}}{r_{1}}=1.
\]
We have known that 
\[
r_{0}\leq r_{1}\leq Cr_{0}.
\]
Once the conclusion were false, we could find a subsequence of $\Sigma_{n}$
(also denoted by $\Sigma_{n}$ ) such that 
\[
\lim_{n\rightarrow\infty}\frac{r_{0}}{r_{1}}=k
\]
 with $0<k<1.$ Then by taking a subsequence further, $\Sigma_{n}$
would converge to some $S_{1}(a)$ with $0<|a|<1$ in $C^{2,\alpha}$
sense globally from Lemma \ref{Blow down by H/2}. Choose $b=\frac{-a}{|a|}$

\begin{align*}
 & \int_{\Sigma_{n}}(H-H_{e})<v_{e},b>_{e}d\mu_{e}\\
= & \int_{\Sigma_{n}^{\frac{2}{H}}}(-\frac{1}{2}h_{\bar{i}\bar{j}}^{2/H}\bar{v}^{\bar{m}}\bar{b}^{\bar{m}}+\frac{1}{2}h_{\bar{i}\bar{l}}^{2/H}\bar{v}^{\bar{l}}\bar{b}^{\bar{i}}-\frac{1}{2}\bar{v}^{\bar{l}}\partial_{\bar{i}}h_{\bar{i}\bar{l}}^{2/H}(\bar{x}^{\bar{m}}-a^{\bar{m}})\bar{b}^{\bar{m}}\\
 & +\frac{1}{2}v^{\bar{l}}\partial_{\bar{l}}h_{\bar{i}\bar{i}}^{2/H}(\bar{x}^{\bar{m}}-a^{\bar{m}})\bar{b}^{\bar{m}})d\bar{\mu}_{e}(\Sigma_{n}^{\frac{2}{H}})+o(1)
\end{align*}
By using the same method as Lemma \ref{H^-1 scale} we have 
\begin{align*}
 & \lim_{n\rightarrow\infty}\int_{\Sigma_{n}^{\frac{2}{H}}}(-\frac{1}{2}h_{\bar{i}\bar{j}}^{2/H}\bar{v}^{\bar{m}}\bar{b}^{\bar{m}}+\frac{1}{2}h_{\bar{i}\bar{l}}^{2/H}\bar{v}^{\bar{l}}\bar{b}^{\bar{i}}-\frac{1}{2}\bar{v}^{\bar{l}}\partial_{\bar{i}}h_{\bar{i}\bar{l}}^{2/H}\bar{x}^{\bar{m}}\bar{b}^{\bar{m}}\\
 & +\frac{1}{2}v^{\bar{l}}\partial_{\bar{l}}h_{\bar{i}\bar{i}}^{2/H}\bar{x}^{\bar{m}}\bar{b}^{\bar{m}})d\bar{\mu}_{e}(\Sigma_{n}^{\frac{2}{H}})\rightarrow0
\end{align*}
as $n\rightarrow\infty.$ And 
\begin{align*}
 & \lim_{n\rightarrow\infty}\int_{\Sigma_{n}^{\frac{2}{H}}}(\frac{1}{2}\bar{v}^{\bar{l}}\partial_{\bar{i}}h_{\bar{i}\bar{l}}^{2/H}a^{\bar{m}}\bar{b}^{\bar{m}}-\frac{1}{2}v^{\bar{l}}\partial_{\bar{l}}h_{\bar{i}\bar{i}}^{2/H}a^{\bar{m}}\bar{b}^{\bar{m}})d\bar{\mu}_{e}(\Sigma_{n}^{\frac{2}{H}})\\
= & \lim_{n\rightarrow\infty}(-|a|)\int_{\Sigma_{n}^{\frac{2}{H}}}(\frac{1}{2}\bar{v}^{\bar{l}}\partial_{\bar{i}}h_{\bar{i}\bar{l}}^{2/H}-\frac{1}{2}v^{\bar{l}}\partial_{\bar{l}}h_{\bar{i}\bar{i}}^{2/H})d\bar{\mu}_{e}(\Sigma_{n}^{\frac{2}{H}})\\
= & -|a|m.
\end{align*}
So we get a contradiction with (\ref{trivial identity}). Now we have
proved Theorem \ref{thm1}.

\paragraph{Proof of Corollary \ref{remove Huang's radius}.}

In Huang's case, for $q\in(\frac{1}{2},1]$, by using (\ref{eq:Constraint equations}),
we get the scalar curvature 
\[
R=O(r^{-2-2q})
\]
where $-2-2q<-3.$ So when $q=1$, actually the metric used by Huang
is $C_{1,1}^{4}$-AF. So we can apply Theorem \ref{thm1} and get
the radius pinching estimate 
\[
r_{1}\leq Cr_{0}
\]
 which is sufficient to prove the uniqueness through Huang's uniqueness
theorem.

\paragraph{Proof of Corollary \ref{remove Nerz's radius}.}

Let $(M,g)$ be $C_{1,\tau}^{4}$-AF manifold, with $m>0.$ By Theorem
1, any stable CMC sphere that separates $\tilde{K}$ from infinity
has 
\[
r_{1}\leq Cr_{0}
\]
which implies the first and the second condition in (\ref{Nerz's three conditions}).
Because the genus is $0$, (\ref{integral estimate of H}) implies
the third condition in (\ref{Nerz's three conditions}). So we can
get the uniqueness from Nerz's uniqueness theorem.

\[
\]
\[
\]

\bibliographystyle{plain}
\bibliography{/Users/apple/Documents/Documents/Mathbib}

\[
\]

\[
\]
 
\[
\]

\end{document}